\documentclass[11pt]{amsart}

\makeatletter
\renewcommand*\env@matrix[1][\arraystretch]{%
  \edef\arraystretch{#1}%
  \hskip -\arraycolsep
  \let\@ifnextchar\new@ifnextchar
  \array{*\c@MaxMatrixCols c}}
\makeatother

\newcommand{\vertiii}[1]{{\left\vert\kern-0.25ex\left\vert\kern-0.25ex\left\vert #1 
    \right\vert\kern-0.25ex\right\vert\kern-0.25ex\right\vert}}

\usepackage{marginnote}
\usepackage{hhline}
\usepackage{amsmath}
\usepackage{amssymb}
\usepackage{amsthm}
\usepackage{amsfonts}
\usepackage{mathtools}
\usepackage{enumerate}
\usepackage{comment}
\usepackage{makeidx}
\usepackage{hyperref}
\usepackage{multirow}
\usepackage[noadjust]{cite}
\usepackage{color}
\usepackage{pst-node}
\usepackage[toc,page]{appendix}
\allowdisplaybreaks
\mathtoolsset{showonlyrefs}

\DeclareMathOperator{\sech}{sech}

\DeclareMathOperator{\sign}{sign}

\DeclareMathOperator{\Sec}{Sec}

\let \sectionsymbol \S
\newcommand{\de}{\delta}
\newcommand{\ga}{\gamma}

\newcommand{\R}{\mathbb{R}}
\renewcommand{\S}{\mathbb{S}}

\newcommand{\tM}{\widetilde{M}}

\newcommand{\p}{\partial}

\newcommand{\oM}{\overline{M}}

\newcommand{\rg}{\rangle}
\renewcommand{\lg}{\langle}

\newcommand{\prl}{\parallel}
\newcommand{\fe}{\varphi}
\newcommand{\ka}{\kappa}

\newcommand{\cA}{\mathcal{A}}

\newcommand{\cF}{\mathcal{F}}

\newcommand{\cI}{\mathcal{I}}
\newcommand{\cK}{\mathcal{K}}

\newcommand{\cR}{\mathcal{R}}

\newcommand{\cU}{\mathcal{U}}
\newcommand{\cV}{\mathcal{V}}
\newcommand{\cW}{\mathcal{W}}
\newcommand{\cX}{\mathcal{X}}
\newcommand{\cY}{\mathcal{Y}}

\renewcommand{\a}{\alpha}
\renewcommand{\b}{\beta}
\newcommand{\g}{\gamma}
\renewcommand{\d}{\delta}
\let\epsilon\varepsilon
\newcommand{\e}{\epsilon}

\newcommand{\G}{\Gamma}

\renewcommand{\r}{\rho}
\newcommand{\rb}{\overline{\rho}}

\renewcommand{\l}{\lambda}

\renewcommand{\th}{\theta}

\newcommand{\w}{\omega}

\theoremstyle{plain}
\newtheorem{theorem}{Theorem}

\newtheorem{proposition}{Proposition}
\numberwithin{proposition}{section}

\newtheorem{lemma}[proposition]{Lemma}

\theoremstyle{definition}

\theoremstyle{remark}

\numberwithin{equation}{section}

\let \o \undefined
\def \o#1{\overline{#1}}

\newcommand{\dr}{\r'}

\newcommand{\A}{\mathcal{A}}

\renewcommand{\e}{\epsilon}

\begin{document}

\title[Boundary Conjugate Points 
  but no Interior Conjugate Points]{Asymptotically Hyperbolic
  Manifolds with Boundary Conjugate Points  
  but no Interior Conjugate Points}
\author[Nikolas Eptaminitakis and C. Robin Graham]{Nikolas Eptaminitakis}
\address{Department of Mathematics, University of Washington,
Box 354350\\
Seattle, WA 98195-4350, USA}
\email{neptamin@uw.edu}

\author[]{C. Robin Graham}
\email{robin@math.washington.edu}



\begin{abstract}
We construct non-trapping asymptotically hyperbolic manifolds with boundary
conjugate points but no interior conjugate points.
\end{abstract}

\maketitle

\thispagestyle{empty}

\begin{center}
{\it Dedicated to the memory of Elias M. Stein}
\end{center}

\bigskip

\section{Introduction}

Let $\oM$ be a smooth compact manifold with boundary with interior $M$ and
let $r$ be a boundary defining function.  
A Riemannian metric $g$ on $M$ is called \textit{asymptotically hyperbolic
(AH)} if $\o{g}:=r^2 g$ extends smoothly to a Riemannian   
metric on $\oM$ and in addition $|dr|_{\o{g}}\equiv1$ on $\p M$.    
As shown in \cite{MR2941112}, AH metrics are complete, with sectional
curvatures approaching $-1$ as $r\to 0$. 
A manifold endowed with an AH metric will also be called asymptotically hyperbolic.
The most important example of an AH manifold is the Poincar{\'e} model of
hyperbolic space with constant sectional curvature $-1$; the underlying
manifold is $\mathbb{B}^{n+1}=\{x\in \R^{n+1}:|x|<1\}$, endowed with the
metric $\displaystyle g=\frac{4\sum_{j=0}^n(dx^j)^2}{(1-|x|^2)^2}$. 

An AH manifold $M$ is called non-trapping if given any compact set
$K\subset M$ and any unit speed geodesic $\g(t)$ there exists $T_{K,\g}$ 
so that $\g(t)\notin K$ for $|t|\geq T$.  It was proved in \cite{MR2941112}
that any geodesic of a non-trapping AH manifold $M$ approaches boundary
points $p^\pm\in \p M$ as $t\to\pm\infty$.  Recall that a Riemannian
manifold is said to have no conjugate points if any non-trivial Jacobi
field along a unit speed geodesic vanishes at most once.  On a non-trapping AH manifold
one can make sense of \textit{boundary conjugate points} along a unit speed
geodesic $\g$: two boundary points $p^+$, $p^-\in\p M$ are called conjugate
along $\g$ if $\lim_{t\to\pm\infty}\g(t)=p^\pm$ and there exists a
non-trivial Jacobi field $Y$ along $\g$ satisfying
$\lim_{t\to\pm\infty}|Y(t)|_g=0$.  We will often 
call the usual conjugate points \textit{interior}, to distinguish
them from boundary conjugate points.  On AH manifolds with non-positive
sectional curvature there are no interior or boundary conjugate points.  We
also mention that a result in \cite{MR0380891} implies that if an AH
manifold has no interior conjugate points then there is no Jacobi field
$Y(t)$ along a unit speed geodesic with the property
$|Y(0)|_g=0=\lim_{t\to\infty}|Y(t)|_g$, that is, no
``interior-boundary'' conjugate points can occur.
In this paper we prove the following:  
\begin{theorem}\label{010319maintheorem}
For any integer $n\geq 1$ there exist smooth non-trapping
asymptotically hyperbolic manifolds of dimension $n+1$ with
boundary conjugate points but no interior conjugate points.
\end{theorem}

Our interest in this question arose in connection with the 
formulation of the definition of a \textit{simple} AH manifold in 
\cite{2017arXiv170905053G}.  Recall 
that one of the equivalent definitions of a simple compact Riemannian
manifold with boundary is that the boundary be strictly convex, the
manifold be non-trapping, and no pair of points  
(in the interior or on the boundary) be conjugate along any geodesic.  
Simple compact manifolds with boundary are the most basic natural setting
for the study of geometric inverse problems.  The paper 
\cite{2017arXiv170905053G} was concerned with extending this study to the 
AH setting, which necessitated among other things formulating an analogous
definition of a simple AH manifold.  In the AH case, convexity  
of the boundary is in a sense automatic:  for any boundary defining
function $r$ and $\e>0$ small enough the sets  
$r\geq \e$ are strictly convex.  The definition in
\cite{2017arXiv170905053G} of a simple AH manifold is that the AH manifold
be non-trapping and without boundary conjugate points.  It was shown that 
these conditions imply that the geodesic flow is Anosov with respect to the
Sasaki metric, which together with the main result of \cite{MR3825848}
implies that there are no interior conjugate points either.  The question
thus arose of whether it is equivalent to assume the manifold is AH,    
non-trapping and without interior conjugate points.  
Theorem~\ref{010319maintheorem} resolves this in the negative.

In the 1970s there was a great deal of interest and activity concerned 
with understanding the relationships between various properties on a
Riemannian manifold such as absence of conjugate points, Anosov geodesic
flow, absence or presence of focal points, and existence of open sets of
strictly positive curvature.  Our construction is inspired by techniques
used in \cite{MR0383294} to construct metrics elucidating the relationships
between some of these properties.  Such questions remain of current
interest; see, for example, \sectionsymbol 2.3 of \cite{guillarmou2019xray}
where methods of \cite{MR0383294} are used to construct an  
asymptotically conic metric on $\R^n$ which has positive curvature on an
open set but no conjugate points.
We start by constructing a non-trapping, complete,
$O(n+1)$-invariant $C^{1,1}$ metric on $\R^{n+1}$ which compactifies to an
AH metric, such that there are no
nontrivial Jacobi fields that vanish twice in the interior but along
radial geodesics there are Jacobi fields that vanish as both   
$t\to \pm\infty$.  Here the $C^{1,1}$ 
regularity implies existence and uniqueness of geodesics; Jacobi fields are
understood in a weak sense.  Our manifold has constant positive sectional
curvature in an open geodesic ball and negative sectional curvature outside
a compact set; when $n=1$, the negative sectional curvature is
constant whereas when $n\geq 2$ this is not the case.  For our purposes,
the size of the set of positive curvature has to be  
carefully chosen: if it is too large, interior conjugate points occur,
whereas if it is too small no boundary conjugate points occur; there is a
critical size for which there exist boundary conjugate points but no
interior ones.  Because of this, our analysis is much
more delicate than that of \cite{MR0383294}, where the conditions are 
open.  Somewhat 
surprisingly, it turns out that for our $C^{1,1}$ metric one can compute
exact formulas for all geodesics, sectional curvatures and Jacobi fields
even though the manifold has non-constant curvature outside any compact set
for $n\geq 2$.  For this reason our $C^{1,1}$ metric may be of more general
interest.     

In the second half of the paper we show that our metric can be approximated
by smooth metrics that still have all the required properties.  As already
hinted, 
these properties are quite unstable under perturbations of the metric:
small variations can result in either presence of interior conjugate
points or absence of boundary ones.  The analogous approximation step in 
\cite{MR0383294} was trivial; any smooth, or even real-analytic, metric 
sufficiently close continued to satisfy the requisite conditions.  We
analyze the \textit{stable} Jacobi fields, defined as those which vanish as
$t\to \infty$.  By careful choice of parameters in our approximating metric
we arrange that there is a stable Jacobi field along radial geodesics 
which also vanishes as $t\to -\infty$ so that 
the corresponding metric has boundary conjugate points.  We then derive a
criterion (Proposition \ref{123018proposition}) in terms of the behavior of
the stable solution for certain second order ODE's that rules out solutions
vanishing twice.  The relevant behavior can be controlled under
perturbations of the metric to rule out interior conjugate points.  Our
argument requires control over third 
derivatives of the stable solutions (two in a parameter and one in the  
time variable) as the approximating metric approaches the $C^{1,1}$ metric,  
for which we have to carry out some rather technical analysis.

The paper is organized as follows. The $C^{1,1}$ metric is constructed in
Section \ref{c1manifoldsection}: in \ref{sec:the_metric} we
define it and state some general properties, in 
\ref{013019curvaturegeodesics} we show explicit formulas for the curvature
along geodesics and in \ref{121118section} we compute formulas for
Jacobi fields and show Theorem \ref{010319maintheorem} in the $C^{1,1}$
case.  In Section \ref{smoothingsection} we prove Theorem  
\ref{010319maintheorem} in the $C^\infty$ case.  We first reduce 
Theorem \ref{010319maintheorem} to three propositions (\ref{exist}, 
\ref{noicpssmall}, \ref{noicpslarge}) 
concerning stable Jacobi fields and absence of conjugate points for the
approximating metrics.  Then we 
carry out the analysis of the derivatives of the stable solutions, prove 
Proposition \ref{123018proposition} which rules out interior conjugate
points, and conclude by proving Propositions~\ref{exist},
\ref{noicpssmall}, \ref{noicpslarge}.   

\medskip
\noindent \textbf{Acknowledgments.}
Research of N.E. was partially supported by the National Science Foundation 
under Grants No. DMS-1800453 and DMS-1265958 of Gunther Uhlmann. It was 
also partially supported by the National Science Foundation under Grant
No. DMS-1440140 while N.E. was in residence at the Mathematical Sciences
Research Institute in Berkeley, California, during the Fall 2019 semester. 

\smallskip

\section{The \texorpdfstring{$C^{1,1}$}{C1,1}
  Metric}\label{c1manifoldsection} 

\subsection{The Metric}\label{sec:the_metric}

We will construct metrics on 
$\R^{n+1}\backslash \{0\}\simeq(0,\infty)_\r\times \S^n$  
of the form 
\begin{equation}\label{eq1} 
g=d\r^2+\A^2(\r) \mathring{g}
\end{equation} 
in polar coordinates that extend smoothly to the origin.
Here $\mathring{g}$ denotes the round metric on $\S^n$ and $\A(\r)$ is a
positive function on $(0,\infty)$ to be chosen appropriately, with
$\cA(\r)=\sin(\r)$ for $\r$ small.  Hence in a neighborhood of the 
origin $g$ is smooth and is isometric to the round metric on 
$\S^{n+1}$.  
Relative to the product decomposition
$\R^{n+1}\backslash \{0\}\simeq(0,\infty)\times \S^n$, the non-zero
Christoffel symbols of $g$ are 
\begin{equation}\label{091218:Christoffel}
	\begin{tabular}{c c c}
		$\G_{\a\b}^0=-\cA(\r)\cA'(\r)\mathring{g}_{\a\b}$, &
          $\G_{\a0}^\g=\cA^{-1}(\r)\cA'(\r)\d_\a^\g$, &
          $\G_{\a\b}^\g=\mathring{\G}_{\a\b}^\g$,  
	\end{tabular}
\end{equation}
where $\mathring{\G}$ are the Christoffel symbols of the round metric and
$\r$ is the 0-th coordinate.\footnote{Throughout this paper, Greek indices run 
  from $1$ to $n$ and Latin indices run from 0 to $n$.} 
The form of the Christoffel symbols implies that for any  
$k=1,\dots,n+1$, $k$-dimensional Euclidean planes passing through the
origin are totally geodesic.  To see this, note that the intersection of
$\S^n\subset \R^{n+1}$ with any $k$-dimensional plane through the origin is
totally geodesic for the round metric, and that in general an embedded
submanifold $M^k\subset \tM^d$ is totally geodesic if and only if in any
coordinate chart $(U,\phi)$ for which $\phi(U\cap M)\subset\{(z,z')\in
\R^{k}\times\R^{d-k}: z'=0\}$, the Christoffel symbols  
satisfy $\G_{ij}^m=0$ on $M\cap U$ for $i,j\leq k$ and all $m\geq k+1$.  As
a special case, lines of the form $\g(t)=t v$ for $v\in \R^{n+1}$ with
Euclidean length 1 are totally geodesic, and in fact they are radial unit
speed geodesics for $g$. 

The curvature tensor of warped product metrics like $g$ can be described
as follows.  This is a special case of Proposition 42, Chapter 7 in
\cite{MR719023}.    
\begin{proposition}\label{042919prop}
	Let $g=d\r^2 +\cA^2(\r)b$, where $\rho\in \R$, $0<\cA\in
        C^\infty(\R)$, and $b$ is 
        a metric on a manifold $B$.  If $R$, $R_b$ denote the Riemannian
        curvature tensors of $g$, $b$, respectively, and $U$, $V$, $W\in  
        \mathfrak{X}(B)$, then
	\begin{enumerate}[(1)]
		\item $R(\p_\r,V)\p_\r=-\cA^{-1}(\r)\cA''(\r) V$
		\item \label{042919item} $R(V,W)\p_\r=0$
		\item $R(\p_\r,V)W=\lg V,W\rg_g \cA^{-1}(\r)\cA''(\r)\p_\r$
		\item $R(V,W)U= R_b(V,W)U-(\cA'(\r))^2\cA^{-2}(\r)(\lg
                  V,U\rg_g W-\lg W,U\rg_g V)$. 
	\end{enumerate}
\end{proposition}
For an $O(n+1)$-invariant metric on $\R^{n+1}$, the sectional 
curvature of a 2-plane $\Pi\subset T_p\R^{n+1}$ at a point $p=\r\w$,
$\r>0$, $\w\in \S^n$, depends only on $\r$ and the angle $\a$ between $\p_\r$
and $\Pi$.  We will denote any such plane by $\Pi_{\r;\cos(\a)}$ and the 
corresponding sectional curvature by $\Sec(\Pi_{\r;\cos(\a)})$.\footnote{We
  use $\Sec$ for sectional curvature, as opposed to $\sec$ which will be
  used for the secant of a real number.}     
Then by Proposition \ref{042919prop} we find, for 2-planes parallel to the radial direction,
\begin{equation}
	\Sec(\Pi_{\r;1})=-\cA(\r)^{-1}\cA''(\r)=:K^\prl(\r).\label{122618curvature}
\end{equation}
Moreover, if $n\geq 2$, for 2-planes normal to the radial direction we have
\begin{equation}
	\Sec(\Pi_{\r;0})=\cA^{-2}(\r)-\cA^{-2}(\r)(\cA'(\r))^2=:K^\perp(\r).\label{eq:Kperp}
\end{equation}
More generally, it follows from (\ref{042919item}) in Proposition
\ref{042919prop} and the symmetries of the curvature tensor that
$R(u,w,u,\p_\r)=0$ for $u$, $w\in \p_\r^\perp$, so for $\a\in[0,\pi/2]$ we
have 
\begin{equation}\label{121818curvature}
	\Sec(\Pi_{\r;\cos(\a)})=\cos^2(\a)K^\parallel(\r)+\sin^2(\a)K^\perp(\r).
\end{equation}
It will later be convenient to use \eqref{121818curvature} to {\it define} 
$\Sec(\Pi_{\r;\cos(\a)})$ for $\cos(\a)\in[-1,0)$ so that 
\eqref{121818curvature} holds for all $\cos(\a)\in[-1,1]$ and $\r>0$.   
{From} \eqref{122618curvature} and the fact that $\cA(\r)=\sin(\r)$ for small
$\r$ it follows that $\A$ solves the equation  
$\cA''(\r)+\Sec(\Pi_{\r;1})\A(\r)=0$ with $\cA(0)=0$ and $\cA'(0)=1$.

\medskip

The previous discussion indicates that the geometry induced by $g$ on
$\R^{n+1}$ is entirely determined by the radial curvature function
$K^\parallel$, thus our goal will be to choose it appropriately. 
We let, for $\r\geq 0$, $r>0$ and $\e\geq 0$,
\begin{equation}\label{eq3}
K_{r,\e}^{\|}(\rho)=
\begin{cases}
1-2\fe(\frac{\rho-r}{\e}) &  \qquad \e>0 \\
1-2H(\rho-r) &  \qquad \e=0,
\end{cases}
\end{equation}
where $\fe \in C^\infty(\R)$ satisfies $0\leq \fe\leq 1$, 
$\fe(\rho)=0$ for $\rho \leq 0$ and $\fe(\rho)=1$ for $\rho \geq 1$, and
$H$ is the Heaviside function:  $H(\r)=0$ if $\r\leq 0$, $H(\r)=1$ if
$\r>0$.  In particular, $K_{r,\e}^{\|}(\r)=1$ for $\r\leq r$.  Observe that 
$K_{r,\e}^{\|}$ is $C^\infty$ if $\e>0$ and is 
piecewise $C^\infty$ if $\e=0$. Moreover, for each $r$, 
$K_{r,\e}^\parallel- K_{r,0}^\parallel\to 0$ in $L^1([0,\infty))$ as 
$\e\to 0$.    

Define $\cA_{r,\e}$ to be the solution to  
\begin{equation}\label{Aeq}
\cA''+K^{\|}_{r,\e}\cA=0,\qquad \cA(0)=0,\quad \cA'(0)=1,  
\end{equation}
where if $\e=0$, $\cA_{r,0}$ is interpreted as a weak solution.  This 
means that it is the unique $C^1$ function satisfying the initial
conditions in $\eqref{Aeq}$ which in addition solves the differential 
equation in the open intervals where $K_{r,0}^\|$ is smooth.  
Observe that for all $\e\geq 0$,    
\begin{equation}\label{Aformula}
\cA_{r,\e}(\rho)=
\begin{cases}
\sin(\rho) & \rho\leq r\\ 
a_+e^{\rho}+a_-e^{-\rho} & \rho\geq r +\e, 
\end{cases}
\end{equation}
where $a_\pm$ depend on $r$, $\e$.   
When $\e=0$, the values of $a_\pm$ are determined
by matching the value and the derivative at $\r=r$ with those of
$\sin(\r)$.  The case $r=\pi/4$ is special in that $a_-=0$:   
\begin{equation}\label{specialA}
\cA_{\pi/4,0}(\r)=\frac{\sqrt{2}}{2}e^{\r-\pi/4}\qquad\qquad \r\geq \pi/4.
\end{equation}
This is fortuitous, because as we will see, $r=\pi/4$ is precisely the
value for which the corresponding metric has boundary conjugate  
points but no interior conjugate points.  It is not true that $a_-=0$ for
other choices of $r$, including the degenerate case $r=0$, $\e=0$, which
corresponds to hyperbolic space.  

The Sturm Comparison Theorem implies that $\cA_{r,\e}>0$ on $(0,\pi)$ and
comparison of the Pr\"ufer angle (Theorem 1.2, p. 210 of  
\cite{MR0069338}) shows that $\cA_{r,\e}'>0$ on $(0,\pi/2)$.  Since  
$\cA_{r,\e}''=\cA_{r,\e}$ on $(r+\e,\infty)$, it follows that
$\cA_{r,\e}>0$ and $\cA_{r,\e}'>0$ on   
$(0,\infty)$ if we require $r+\e<\pi/2$, which we do henceforth.  In
particular, $a_+>0$ in \eqref{Aformula}.  
Ultimately we will only care about $r$ near $\pi/4$ and $\e$ near $0$.   

We will denote by $g_{r,\e}$ the metric given by \eqref{eq1} with $\cA$
replaced by $\cA_{r,\e}$.  
So $g_{r,\e}$ restricted to the geodesic ball $B_r(0)$ centered at 
the origin is isometric to the corresponding geodesic ball in $\S^{n+1}$,
and in particular the sectional curvatures of $g_{r,\e}$ are all equal to 1
for $\r<r$.  The sectional curvature of 2-planes parallel to the radial
direction is $-1$ for $\r>r+\e$, but not for other 2-planes if $n\geq 2$.  
The metric $g_{r,\e}$ is asymptotically hyperbolic (but only
$C^{1,1}$ if $\e=0$) if 
$\R^{n+1}$ is radially compactified with defining function $e^{-\r}$ for
the boundary at infinity.  In particular, $g_{r,\e}$ is complete.

Our goal in Section~\ref{c1manifoldsection} is to show that 
$g_{\pi/4,0}$ satisfies all of the properties stated in 
Theorem~\ref{010319maintheorem} except for smoothness.    

\subsection{Geodesics and Sectional
  Curvature}\label{013019curvaturegeodesics}

Since $g=g_{r,\e}$ is at least $C^{1,1}$, it determines geodesics of class
at least $C^{2,1}$.  Provided a unit speed geodesic $\g(t)$ of $g$
is not radial, $\g(0)$ and $\g'(0)$ 
determine a unique 2-plane through the origin denoted by $\Sigma_\g$; as
mentioned earlier, $\Sigma_\g$ is totally geodesic and hence $\g$ is
entirely contained in it.  For radial geodesics $\g$, we will write 
$\Sigma_\g$ for any 2-plane containing $\g$.

To study any unit speed geodesic $\g$ it is sufficient to work
in $\Sigma_\g$ with induced metric  
$g\big|_{\Sigma_\g}=d\r^2+\cA^2(\r)d\th^2$, where $\cA=\cA_{r,\e}$.     
According to \eqref{091218:Christoffel}, $\rho(t):=\rho(\g(t))$ satisfies
the equation
\begin{equation}\label{rhoequation}
\r''=\cA^{-1}(\r)\cA'(\r)\left(1-(\r')^2\right).
\end{equation}  
If $\g$ is not radial, the initial conditions take the 
form $\rho(0)=s>0$, $\rho'(0)=v$ with $|v|<1$.  
It is evident from \eqref{Aformula} that there is $a=a_{r,\e}>0$ so that  
$\cA^{-1}(\r)\cA'(\r)\geq a$ for all $\r>0$.  A comparison theorem
(e.g. Theorem 11.XVI of \cite{MR1629775}) implies that $\r(t)\geq 
\rb(t)$ for all $t\in \R$, where $\rb$ is the solution of  
\begin{equation}\label{constgeodeq}
\r''=a\left(1-(\r')^2\right)
\end{equation}  
satisfying the same initial conditions.  Equation \eqref{constgeodeq} is
separable for $\r'$; the solution is
\begin{equation}\label{rhosolution}
\rb(t)=s
+a^{-1}\log\Big(\tfrac12\big((1+v)e^{at}+(1-v)e^{-at}\big)\Big).   
\end{equation}
It follows in particular that $\r(t)\to \infty$ as $t\to \pm \infty$ so 
that $g_{r,\e}$ is nontrapping.   
Since $\r''>0$, $\r$ achieves its minimum at a unique time which we
take to be $t=0$.  The corresponding point is the closest point on 
$\g$ to the origin, whose distance to the origin we write $s$.  We denote  
this solution by $\r_{s,r,\e}$; it is thus the solution to 
\eqref{rhoequation} with $\cA=\cA_{r,\e}$ and with initial
conditions 
$\rho(0)=s>0$, $\rho'(0)=0$.  For a radial geodesic the distance to the
origin is $s=0$ and the corresponding
solution is $\r_{0,r,\e}(t)=t$.  We denote by $\g_{s,r,\e}$ any unit speed
geodesic with radial coordinate function $\r_{s,r,\e}$.

If $s< r$, then $\g_{s,r,\e}$ intersects the geodesic ball $B_r(0)$  
where the curvature is 1 and $\cA(\rho) = \sin(\rho)$.   
In this case, it is easily checked
by directly verifying \eqref{rhoequation} and the initial 
conditions that 
\begin{equation}\label{rhotsmall}
\r_{s,r,\e}(t)=\arccos\big(\!\cos(s)\cos(t)\big).
\end{equation}
This holds up to the time $t$ such that $\r_{s,r,\e}(t)=r$.  We denote this
time by $\ell_r(s)$; geometrically this is the distance between
$\g_{s,r,\e}(0)$ and $\p B_r(0)$ and clearly it is given by  
\begin{equation}\label{ell}
\ell_r(s) = \arccos\left(\frac{\cos(r)}{\cos(s)}\right).
\end{equation}
For future reference note that 
\begin{equation}\label{rhoprime}
\r_{s,r,\e}'(\ell_r(s))= 
\frac{\sqrt{\cos^2(s)-\cos^2(r)}}{\sin(r)}.
\end{equation}
This also has a geometric interpretation:  since $\p_\r$ and $\g'$ are unit
vectors, 
$\r_{s,r,\e}'(\ell_r(s))=\lg \g'(\ell_r(s)),\p_\r\rg= \cos{(\a)}$, 
where $\a$ is the the angle between $\g'(t)$ and $\p_\r$ when
$t=\ell_r(s)$, i.e. where $\r(t)=r$.  The above  formulas for 
$\r_{s,r,\e}(t)$, $\ell_r(s)$ and $\r_{s,r,\e}'(\ell_r(s))$ can also be 
derived directly via the geometry of $\S^2$.

Our primary focus in Section~\ref{c1manifoldsection} is  the 
case $r=\pi/4$, $\e=0$.  We suppress these subscripts, so for instance  
subsequently we write $g=g_{\pi/4,0}$, $\g_s=\g_{s,\pi/4,0}$,    
$\r_s(t)=\r_{s,\pi/4,0}(t)=\r(\g_s(t))$, $\ell(s)=\ell_{\pi/4}(s)$.  Note
that for $r=\pi/4$, \eqref{rhoprime} reduces to 
$\rho_s'(\ell(s))=\sqrt{\cos(2s)}$.   

When $r=\pi/4$ and $\e=0$, \eqref{specialA} shows that
\eqref{rhoequation} for $\r>\pi/4$   
reduces to \eqref{constgeodeq} with $a=1$.  The initial conditions for
$\r_s$  are $\r_s(\ell(s))=\pi/4$, $\r_s'(\ell(s))=\sqrt{\cos(2s)}$.  The
solution is given by \eqref{rhosolution} with $t$ replaced by $t-\ell(s)$.
It can be written in the form
\[
\r_s(t)=\pi/4 + \log F(t,s)\qquad t\geq \ell(s),
\]
where
\begin{equation}\label{F}
F(t,s)=\cosh(t-\ell(s))+\sqrt{\cos(2s)}\sinh(t-\ell(s)).
\end{equation}
For $t\leq -\ell(s)$ one has $\r_s(t)=\r_s(-t)$.  

Equation \eqref{121818curvature} expresses the sectional curvatures of $g$
in terms of the distance $\r$ to the origin and the angle $\a$ between $\p_\r$
and the plane $\Pi$.  In our subsequent analysis of Jacobi fields, the 
sectional curvature for $g_{r,\e}$ along a geodesic   
$\g_{s,r,\e}(t)$ of the plane spanned by $\g_{s,r,\e}'(t)$ and a vector
normal to $\Sigma_{\gamma_{s,r,\e}}$ will play a fundamental role.  Since 
$\rho_{s,r,\e}'(t)=\cos(\angle(\g_{s,r,\e}'(t),\p_\r))$,   
the sectional curvature of interest is
\[
K_{s,r,\e}(t):=\Sec(\Pi_{\r_{s,r,\e}(t);\r_{s,r,\e}'(t)}).   
\]
As usual we write $K_s=K_{s,\pi/4,0}$.

\begin{lemma}\label{122818lemma}
Suppose $n\geq 2$ and $0\leq s< \pi/4$.  Then
\begin{equation}\label{121618curvature}
K_s(t)=
\begin{cases}
1, 				& 0\leq |t|< \ell(s)\\
-1+4\sin^2(s)F^{-4}(|t|,s), &|t|>\ell(s) 
\end{cases} 
	\end{equation}
\end{lemma}

\begin{proof}		
For $\r>\pi/4$ we have $K^\parallel(\r)=-1$ and
$K^\perp(\r)=-1+2e^{-2(\r-\pi/4)}$. 
Hence, for $|t|> \ell(s)$  \eqref{121818curvature} yields
\begin{align}\label{121718curvature} 
\hspace*{-.2 in}
\Sec(\Pi_{\r_s(t);\r_s'(t)})
=&(\dr_s(t))^2(-1)+(1-(\dr_s(t))^2)\left(-1+2e^{-2(\r_s(t)-\pi/4)}\right)\nonumber\\     
=&-1+4\sin^2(s)F^{-4}(|t|,s).  
\end{align}
\end{proof}

There is a similar analysis for geodesics that do not intersect the
geodesic ball $B_{\pi/4}(0)$.  This time \eqref{specialA} holds
along the whole geodesic, and the solution \eqref{rhosolution}
of the geodesic equation satisfying the initial conditions $\rho(0)=s$,
$\rho'(0)=0$ is $\rho_s(t)=s+\log(\cosh(t))$.  Repeating the computation 
\eqref{121718curvature} yields the following.
\begin{lemma}\label{012719lemma}
Suppose $n\geq 2$ and $s\geq \pi/4$.  Then 
\begin{equation}\label{012319curvature}
K_s(t)=-1+2e^{-2s+\pi/2}\sech^4(t),\quad t\in \R.
\end{equation}
\end{lemma}

In the analysis above we have only used formulas for the
radial coordinate of geodesics. 
We summarize them here and for completeness also provide the angular
coordinate $\th$ for a geodesic, even though it will not play a role in 
this paper. 
With $F(t,s)$ as in \eqref{F}, 
\begin{equation}\label{0318geodesic}
\r_s(t)=\begin{cases}\displaystyle\arccos(\cos(t)\cos(s)),
&0\leq s< \pi/4, \,|t|\leq \ell(s)\\ 
\log\big(F(|t|,s)\big)+\pi/4, &0\leq s< \pi/4, \,|t|> \ell(s)\\
\log(\cosh(t))+s, & s\geq \pi/4, \,t\in \R 
\end{cases}.
\end{equation}
Setting
\begin{equation}
\th_s(t):=\begin{cases}\displaystyle\arcsin\left(\frac{\sin(t)}{\sqrt{1-\cos^2(s)\cos^2(t)}}\right),   
&0\leq s<  \pi/4, \\ 
&|t|\leq \ell(s)\\
\displaystyle\text{sgn}(t)\left(\frac{2
\sin (s) \sinh \left(\left| t\right|
-\ell(s)\right)}{F(|t|,s)}+\arcsin\big(\sqrt{1-\tan^2(s)}\big)\right),
&0\leq s <\pi/4, \\ 
&|t|> \ell(s)\\
\sqrt{2}\tanh(t)e^{-s+\pi/4}, & s \geq  \pi/4, \\ 
&t\in \R
\end{cases},
\end{equation}
the curve $(\r_s(t),\th_s(t))$ on $\Sigma_{\g_s}$ satisfies
the geodesic equation for each $s\geq 0$. 
Any other geodesic on $\Sigma_{\g_s}$ can be obtained by translation in $\th$.

\subsection{Analysis of Jacobi Fields}\label{121118section}

In this subsection we first identify the scalar equations solved by 
normal Jacobi fields for $g_{r,\e}$.  Then we  
compute explicitly the normal Jacobi
fields of the $C^{1,1}$ metric $g=g_{\pi/4,0}$ and show
that $(\R^{n+1},g)$ has no interior conjugate points 
but has boundary conjugate points.  

The following general fact can be proved using the Gauss and Codazzi equations:
\begin{proposition}\label{propositiontotallygeodesic}
Let $(M,g_M)$ be a totally geodesic submanifold of a Riemannian manifold
$(\tM,g_{\tM})$.  Let $\g$ be a geodesic contained in $M$ and $Y$ be a
normal $g_{\tM}$-Jacobi field along $\g$.  When $Y$ is decomposed as
$Y=Y_1+Y_2$, where $Y_1$ is everywhere tangent and $Y_2$ is everywhere
normal to $M$, then $Y_2$ is a Jacobi field in $\tM$ and $Y_1$ is a Jacobi
field in both $M$ and $\tM$.
\end{proposition}

\noindent
Proposition \ref{propositiontotallygeodesic} implies that to analyze normal
Jacobi fields along a geodesic $\g$ of $g_{r,\e}$, it is enough
to analyze separately Jacobi fields tangent and normal to $\Sigma_\g$.

Consider first a geodesic $\g\subset \Sigma_\g$ and a Jacobi field $Y(t)$
normal to $\g$ but tangent to $\Sigma_\g$ of the form $Y(t)=\cY(t)E(t)$,
where $E(t)$ is a parallel vector field along $\g$ and $\cY$ is real  
valued.  Since $Y(t)$ is a Jacobi field in the 2-dimensional manifold 
$\Sigma_\g$ and the radial vector field is parallel to $\Sigma_\g$, 
$\cY(t)$ solves the scalar Jacobi equation
\begin{equation}\label{jacobiparallel}
	\cY''(t)+K_{r,\e}^\parallel(\r_{s,r,\e}(t))\cY(t)=0.
\end{equation}

Next consider Jacobi fields along a geodesic  
$\g$ that are orthogonal to the plane $\Sigma_\g$.  The following lemma
reduces the problem to the study of scalar equations.   

\begin{lemma}\label{032219lemma}
Let $n\geq2$, $\g\subset \Sigma_\g $ be a unit speed
geodesic for $g_{r,\e}$ and $Y\perp \Sigma_\g$ a Jacobi field along it.  
Then $Y$ satisfies the scalar Jacobi equation 
\begin{equation}
D_t^2Y(t)+K_{s,r,\e}(t)Y(t)=0. 
\end{equation}
\end{lemma}
\begin{proof}
It is sufficient to show that $R(\g'(t),Y(t))\g'(t)=a(t)Y(t)$, 
$t\in \R$, for some scalar function $a(t)$; then necessarily  $a(t)=K_{s,r,\e}(t)$, since the plane 
determined by $\g'(t)$ and $Y(t)$ is 
is of the form $\Pi_{\r_{s,r,\e}(t);\r'_{s,r,\e}(t)}$.  
The statement is local, so we can use polar coordinates $(\r,\th)$
on $\Sigma_\g$ to write $\g'(t)=\l(t)\p_\r+\mu(t)\p_\th$. 
This implies 
\begin{align}
R(\g',Y)\g'=&\l^2R(\p_\r,Y)\p_\r+\l\mu(R(\p_\r,Y)\p_\th+R(\p_\th,Y)\p_\r)+\mu^2R(\p_\th,Y)\p_\th. 
\end{align}
By Proposition \ref{042919prop}, for the first term we have
$R(\p_\r,Y(t))\p_\r=-\cA''(t)/\cA(t) Y(t)$, the second term
vanishes and for the third we have  
\begin{equation}
R(\p_\th,Y(t))\p_\th=R_{\S^n}(\p_\th,Y(t))\p_\th-(\cA'(t))^2/\cA^2(t)|\p_\th|^2_gY(t).
\end{equation}
Now $R_{\S^n}(\p_\th,Y(t))\p_\th$ $=Y(t)$ since $\S^n$ has constant
sectional curvature $1$, so the lemma is proved. 
\end{proof}

So if we take $Y(t)$ as in Lemma \ref{032219lemma} of the form $Y(t)=\cY(t)E(t)$, where $E(t)$ is a 
parallel vector field along $\g$, then $\cY$ solves the equation
\begin{equation}\label{jacobiperp}
\cY''(t)+K_{s,r,\e}(t)\cY(t)=0 
\end{equation}
where 
\begin{equation}\label{KY}
K_{s,r,\e}(t)
=(\rho_\mu'(t))^2 K_{r,\e}^\|(\rho_\mu(t)) + \big(1-(\rho_\mu'(t))^2\big)
K^\perp_{r,\e}(\rho_\mu(t)),\quad \mu = (s,r,\e)
\end{equation}
and $K^\perp_{r,\e}(\rho)$ is given by \eqref{eq:Kperp} with $\cA$
replaced by $\cA_{r,\e}$.  Note that
$K_{s,r,\e}=K_{r,\e}^\parallel\circ \r_{s,r,\e}$ for $s=0$.  So for radial
geodesics the equations \eqref{jacobiparallel} and \eqref{jacobiperp}
for Jacobi fields tangent and normal to $\Sigma_\g$ coincide.  

We write $\cU^\parallel_{s,r,\e}(t)$, $\cV_{s,r,\e}^\parallel(t)$ for the 
solutions (weak solutions if $\e=0$) of \eqref{jacobiparallel} with the
initial conditions 
$\cU^\parallel_{s,r,\e}(0)=1$, $\cU^\parallel_{s,r,\e}{}'(0)=0$ and 
$\cV^\parallel_{s,r,\e}(0)=0$, $\cV^\parallel_{s,r,\e}{}'(0)=1$.   
Likewise we write $\cU_{s,r,\e}^\perp(t)$, $\cV_{s,r,\e}^\perp(t)$ for the
solutions of \eqref{jacobiperp} satisfying 
$\cU^\perp_{s,r,\e}(0)=1$, $\cU^\perp_{s,r,\e}{}'(0)=0$ and 
$\cV^\perp_{s,r,\e}(0)=0$, $\cV^\perp_{s,r,\e}{}'(0)=1$.  
And once again we suppress $(r,\e)$ when $(r,\e)=(\pi/4,0)$.  

Now we solve \eqref{jacobiparallel}, \eqref{jacobiperp} for 
$(r,\e)=(\pi/4,0)$, beginning with \eqref{jacobiparallel}.   
If $s\geq \pi/4$ then $K^\parallel=-1$, so  
\[
\cU^\prl_s(t)=\cosh(t),\qquad \cV^\prl_s(t)=\sinh(t)\qquad\qquad  
s\geq \pi/4.
\]
If $s<\pi/4$ then $K^\parallel(\r_s(t))$ has a jump discontinuity at
$|t|=\ell(s)$ and the solutions must be $C^1$ across the jump.  It is   
easily verified that 
\begin{equation}\label{Ztildeformula}
\cU^\prl_s(t)=  
\begin{cases}
\displaystyle \cos(t),  &|t|\leq \ell(s)\\
\cos(\ell(s))\cosh(|t|-\ell(s))-\sin(\ell(s))\sinh(|t|-\ell(s)),
&|t|> \ell(s) 
\end{cases},
\end{equation}
\begin{equation}\label{Vformula}
\cV^\prl_s(t)=
\begin{cases}
\displaystyle \sin(t),  &|t|\leq \ell(s)\\
\sign(t)\big(\sin(\ell(s))\cosh(|t|-\ell(s))+\cos(\ell(s))\sinh(|t|-\ell(s))\big), 
&|t|> \ell(s)  
\end{cases}.
\end{equation}

Recall that $\ell(0)=\pi/4$.  
So $\cU^\prl_0(t)=\frac{\sqrt{2}}{2}e^{-(|t|-\pi/4)}$ for $|t|>\pi/4$.  
The corresponding Jacobi field vanishes as $|t|\to \infty$.   
Hence $g$ has boundary conjugate points along radial geodesics.   

\begin{lemma}\label{122918gulliver}
Let $\g$ be a unit speed geodesic for $g_{\pi/4,0}$ contained in a
2-dimensional plane $\Sigma_\g$ through the origin.  
Any non-trivial Jacobi field $Y(t)$ normal to $\g$ and tangent to
$\Sigma_\g$ vanishes at most once.  
\end{lemma}
\begin{proof}
We claim that for any $s\geq 0$, $\cU^\prl_s$ is a positive 
solution of \eqref{jacobiparallel}.  This is clear 
when $s\geq \pi/4$ where $\cU^\prl_s(t) = \cosh(t)$.  For $s<\pi/4$ it  
follows from \eqref{Ztildeformula} and the fact that 
$\sin(\ell(s))\leq \sqrt{2}/2 \leq \cos(\ell(s))$ (recall \eqref{ell}).  
Now the usual Sturm Separation Theorem is valid for an ODE of the
form $\cY''(t)+k(t)\cY(t)=0$, where $k$ is integrable and real
valued and the derivatives are interpreted in a weak sense (see,
e.g. comment in \cite{MR0069338}, p. 208). %
Thus no non-trivial solution of \eqref{jacobiparallel} can vanish
twice.   
\end{proof}

Recall that $K_s(t)$ is identified in Lemmas~\ref{121618curvature} 
and \ref{012319curvature}.  We were astonished to find that 
the scalar Jacobi equations 
\begin{equation}\label{121618jacobi}
\cY''(t)+K_s(t)\cY(t)=0
\end{equation}
can be solved explicitly.  To do so,
note first that for all $t$ if $s\geq \pi/4$ and for $|t|>\ell(s)$  
if $s<\pi/4$, $K_s(t)$ has the form 
$K_s(t)=-1+f^{-4}(t)$, where $f''(t)-f(t)=0$.
Observing that for any $t_0$ such that $f(t_0)\neq 0$ one has
$\big(\frac{\sinh(t-t_0)}{f(t_0)f(t)}\big)'=f^{-2}(t)$, it is easy to check
that $\cY(t)=f(t)b\big(\frac{\sinh(t-t_0)}{f(t_0)f(t)}\big)$ with
$b''(x)+b(x)=0$ is the general solution of the equation  
$\cY''(t)+(-1+f^{-4}(t))\cY(t)=0$.  

For each $s>0$ we identify the solutions $\cU^\perp_s$ and $\cV^\perp_s$ of  
\eqref{121618jacobi}. 
For $s\geq \pi/4$ we take $t_0=0$ and obtain  
\begin{equation}\label{sbig}
\begin{aligned}
\cU^\perp_s(t)&=\cosh(t)\cos\big(\sqrt{2}e^{-s+\pi/4}\tanh(t)\big)\\
\cV^\perp_s(t)&=\frac{\sqrt{2}}{2}e^{s-\pi/4}\cosh(t)\sin\big(\sqrt{2}e^{-s+\pi/4}\tanh(t)\big). 
\end{aligned}
\end{equation}
For $0<s<\pi/4$ we take $t_0=\pm \ell(s)$ and obtain 
\begin{equation}\label{190519solutionz}
\cU^\perp_s(t)=
\begin{cases}
\displaystyle \cos(t),  &|t|\leq  \ell(s)\\
\frac{\sqrt{2}}{2}\csc(s)F(|t|,s)\cos(\Theta(|t|,s)) &|t|> \ell(s) 
\end{cases},\qquad 0< s<\pi/4,
\end{equation}
where $\displaystyle \Theta(t,s):=2 \sin (s) \sinh(t-\ell(s)) F^{-1}(t,s) +\arccos(\tan (s))$ 
and $F(t,s)$ is as in \eqref{F}.  
Note here that $\cos(x+\arccos(\tan (s)))$ is a solution of $b''(x)+b(x)=0$.
Also 
\begin{equation}\label{190519solution}
\cV^\perp_s(t)= 
\begin{cases}
\displaystyle \sin(t),  &|t|\leq \ell(s)\\
\sign(t)\frac{\sqrt{2}}{2}F(|t|,s)\sin(\Theta(|t|,s)) &|t|> \ell(s) 
\end{cases}, \quad 0< s<\pi/4.   
\end{equation}
We remark that these solutions extend smoothly to $s=0$ and 
$\cU^\perp_0=\cU^\prl_0$, $\cV^\perp_0=\cV^\prl_0$.  This is clear for
$\cV_s$, but for $\cU_s$ requires evaluating the indeterminant expression
appearing in \eqref{190519solutionz}.

\begin{lemma}\label{121618lemma}
Let $n\geq 2$ and $\g$ be a unit speed geodesic for $g_{\pi/4,0}$ contained
in a 2-dimensional plane $\Sigma_\g$ through the origin.  
Any nontrivial Jacobi field $Y(t)$ along $\g$ normal to $\Sigma_\g$ 
vanishes at most once. 
\end{lemma}
\begin{proof}
For radial geodesics the proof of Lemma~\ref{122918gulliver} applies since
the equations for Jacobi fields tangent and normal to $\Sigma_\g$ coincide.   

We claim that $\cU^\perp_s$ is everywhere positive for any $s> 0$.  For
$s\geq \pi/4$ this is clear from \eqref{sbig} since 
$|\sqrt{2}e^{-s+\pi/4}\tanh(t)|\leq \sqrt{2}<\pi/2$.  
For $0<s<\pi/4$, according to \eqref{190519solutionz} it suffices to show 
that $0<\Theta(t,s)<\pi/2$ for $t\geq \ell(s)$.  It is easily verified
that for $0<s<\pi/4$ one has $\p_t\Theta(t,s)>0$ for $t\geq \ell(s)$.  So 
for each $s$, $\Theta(t,s)$ strictly increases from a minimum of 
$\arccos(\tan(s))$ at $t=\ell(s)$ to a limit of  
\begin{equation}\label{Thetainfinity}
\Theta_\infty(s):=\lim_{t\to \infty}\Theta(t,s)=
\arccos(\tan(s))+\frac{2\sin(s)}{1+\sqrt{\cos(2s)}}.
\end{equation}
A straightforward calculation shows that 
\begin{equation}\label{derivTheta}
\p_s\Theta_\infty(s)=\frac{-2\sin^2s}
{\cos(s)\left(1+\sqrt{\cos(2s)}\right)^2}
\qquad 0\leq s<\pi/4. 
\end{equation}
So $\Theta_\infty(s)$ strictly decreases from a maximum of $\pi/2$ at $s=0$ to a
minimum of $\sqrt{2}$ at $s=\pi/4$.  Thus
$0<\Theta(t,s)<\pi/2$ for $0<s<\pi/4$ and $t\geq\ell(s)$.   

Once again the result now follows from the Sturm Separation Theorem.
\end{proof}

\begin{proof}[Proof of Theorem \ref{010319maintheorem}, $C^{1,1}$ metric]
We have already noted that $g=g_{\pi/4,0}$ is non-trapping and has 
boundary conjugate points along radial geodesics.  If $Y$ is a normal
Jacobi field along a unit speed geodesic $\g\subset \Sigma_\g$, write  
$Y=Y_1+Y_2$ with $Y_1$ tangent to $\Sigma_\g$ and $Y_2$ normal to it, as in
Proposition \ref{propositiontotallygeodesic}.  If $Y$ vanishes twice, so do
$Y_1$ and $Y_2$.  Lemmas \ref{122918gulliver} and
\ref{121618lemma} imply that both $Y_1$ and $Y_2$ vanish identically, so 
$g$ has no interior conjugate points.  
\end{proof}

\section{Smooth Perturbation}\label{smoothingsection}

In this section we show that we can find $(r,\e)$ near $(\pi/4,0)$ with
$\e>0$ so that $g_{r,\e}$ has no interior conjugate points but has boundary
conjugate points along radial geodesics, thus proving
Theorem~\ref{010319maintheorem}.  First we outline the argument.  Our  
analysis will focus on the decaying (also called stable) solutions of the  
Jacobi equations \eqref{jacobiparallel}, \eqref{jacobiperp}.  As we
argue below, since   
$K^\prl_{r,\e}(\rho_{s,r,\e}(t))$ and $K_{s,r,\e}(t)$ are asymptotic to
$-1$ as $t\to\infty$, there are unique solutions $\cY_{s,r,\e}^\prl(t)$, 
$\cY_{s,r,\e}^\perp(t)$ to 
\eqref{jacobiparallel}, \eqref{jacobiperp}, resp., such that  
$\lim_{t\to \infty}e^t\cY_{s,r,\e}^\prl(t)=1$, 
$\lim_{t\to \infty}e^t\cY_{s,r,\e}^\perp(t)=1$.  
For $K^\prl_{r,\e}(\rho_{s,r,\e}(t))$ this is clear 
since $K^\prl_{r,\e}(\r)=-1$ for $\r$ large.
Of course for $s=0$ we have $\cY_{0,r,\e}^\prl=\cY_{0,r,\e}^\perp$ since
$K^\prl_{r,\e}\circ\rho_{0,r,\e}=K_{0,r,\e}$.   
We will show that $\cY_{0,r,\e}^\prl(0)\neq 0$ for $(r,\e)$ sufficiently
near $(\pi/4,0)$.  
If $\cY_{0,r,\e}^\prl{}'(0)=0$ and $E(t)$ is a non-zero parallel vector
field along $\g_{0,r,\e}$, then $\cY_{0,r,\e}^\prl(|t|)E(t)$ is a 
nontrivial Jacobi field 
which decays as $t\to \pm \infty$.  The corresponding metric $g_{r,\e}$
therefore has boundary conjugate points along radial geodesics. 
We will prove 

\begin{proposition}\label{exist}
Any neighborhood of $(\pi/4,0)$ contains a point $(r,\e)$ with $\e>0$ so
that $\cY_{0,r,\e}^\prl{}'(0)=0$. 
\end{proposition}

\noindent
It then remains to show that the corresponding metric $g_{r,\e}$ 
has no interior conjugate points.  We will do this via the
following two propositions.

\begin{proposition}\label{noicpssmall}
There exist a neighborhood $U$ of $(\pi/4,0)$ and $\sigma>0$ such that if   
$(r,\e)\in U$ and $\cY_{0,r,\e}^\prl{}'(0)=0$, then $g_{r,\e}$ has no
interior conjugate points along any geodesic $\g_{s,r,\e}$ with  
$0\leq s\leq \sigma$.  
\end{proposition}

\begin{proposition}\label{noicpslarge}
For every $\sigma>0$, there exists a neighborhood $V$ of $(\pi/4,0)$ so
that if $(r,\e)\in V$, then $g_{r,\e}$ has no interior conjugate points
along any geodesic $\g_{s,r,\e}$ with $s\geq \sigma$.  
\end{proposition}

\noindent 
Theorem~\ref{010319maintheorem} reduces to these three propositions:
\begin{proof}[Proof of Theorem \ref{010319maintheorem}.]
Choose $U$ and $\sigma$ as in Proposition~\ref{noicpssmall}.  Then choose
$V$ as in Proposition~\ref{noicpslarge} corresponding to this $\sigma$.  
Proposition~\ref{exist} asserts 
that there is $(r,\e)\in U\cap V$ with $\e>0$ so that
$\cY_{0,r,\e}^\prl{}'(0)=0$.  The metric $g_{r,\e}$ then has boundary
conjugate points but no interior conjugate points, and, as before, it is non-trapping.   
\end{proof}

Note that by successively shrinking the neighborhoods, one obtains
a sequence of metrics $g_{r_j,\e_j}$ with $\e_j>0$ and 
$(r_j,\e_j)\to (\pi/4,0)$ such that each $g_{r_j,\e_j}$ has boundary
conjugate points but no interior conjugate points.  The proof actually
shows that for each $\e$ sufficiently small, there is $r_\e$ 
so that $g_{r_\e,\e}$ has boundary conjugate points but no interior
conjugate points. 

Continuity as $\e\to 0$ of solutions of  
\eqref{jacobiparallel}, \eqref{jacobiperp} and of various of their
derivatives in $s$ and $t$ are essential to the proofs of  
Propositions~\ref{exist}, \ref{noicpssmall}, \ref{noicpslarge}.
This is a singular limit, as the functions 
$K^\prl_{r,\e}(\rho_{s,r,\e}(t))$ and $K_{s,r,\e}(t)$ develop jump
singularities as $\e\to 0$.  
We have had to do quite a bit of work to prove
the necessary continuity properties.  
We present this continuity analysis next and afterwards return to        
the proofs of  Propositions~\ref{exist}, \ref{noicpssmall},
\ref{noicpslarge}.  We begin the analysis by formulating some general
results on ODE: Propositions \ref{smooth}-\ref{gronwallnonlinear},    
that we will apply in our setting.

Let $\cF:\R^d\to\R^d$ be a vector field.  Suppose that $\cF$ is continuous and 
piecewise $C^\infty$:  there is a smooth hypersurface $S\subset \R^d$  
locally dividing $\R^d$ into two open subsets $\cU_+$, $\cU_-$ and two 
smooth vector fields $\cF_+$, $\cF_-$ on $\R^d$ so that $\cF=\cF_\pm$ on $\cU_\pm$
and $\cF_+=\cF_-$ on $S$.  Consider the integral curves of $\cF$, which solve the  
ODE
\[
x'(t)=\cF(x(t)),\qquad x(0)=s.
\]
Since $\cF$ is Lipschitz, for each $s$ there exists a unique solution
$x(s,t)$, and $x(s,t)$ and $x'(s,t)$ are jointly continuous.  
We assume
below that the solutions exist on all the time intervals considered. 

\begin{proposition}\label{smooth}
Suppose that $(s_0,t_0)$ has the property that 
$x(s_0,0)$ and $x(s_0,t_0)$ are on opposite sides of $S$, and 
the curve $t\mapsto x(s_0,t)$, $0< t < t_0$,  
crosses $S$ exactly once and does so transversely.  There is a smooth
function $T(s)$ defined for $s$ in a neighborhood $\cV$ of $s_0$ such that
$0<T(s)<t_0$ and $x(s,T(s))\in S$.  The restrictions of $x(s,t)$ to the two 
sets $\{(s,t):s\in \cV, \,0\leq t\leq T(s)\}$ and  
$\{(s,t):s\in \cV,\, T(s)\leq t\leq t_0\}$ are $C^\infty$.  
\end{proposition}
\begin{proof}
Suppose $x(s_0,0)\in \cU_-$ and $x(s_0,t_0)\in \cU_+$.  The curve 
$t\mapsto x(s_0,t)$ is an integral curve of $\cF_-$ up until the time that 
$x(s_0,t)\in S$.  Since $\cF_-$ is smooth, its integral curves are smooth
functions of $(s,t)$.  Since the crossing is transverse, 
there is a unique smooth function $T(s)$ defined for $s$ near $s_0$ by the   
condition that $x(s,T(s))\in S$.  The map $s\mapsto x(s,T(s))$ is
smooth from a neighborhood of $s_0$ to $S$, and $x(s,t)$ is smooth for
$t\leq T(s)$.  For $s$ near $s_0$, the curve 
$t\mapsto x(s,t)$, $t\geq T(s)$ is an integral curve of $\cF_+$ whose initial 
point $x(s,T(s))$ depends smoothly on $s$.  By smoothness of the integral 
curves of $\cF_+$, it follows that $x(s,t)$ is smooth for $t\geq T(s)$.  
\end{proof}

\begin{proposition}\label{c1}
In the setting of Proposition~\ref{smooth}, $x(s,t)$ is jointly $C^1$ 
in a neighborhood of $(s_0,T(s_0))$.    
\end{proposition}
\begin{proof}
We know that $x'(s,t)$ is continuous, and that $\p_sx(s,t)$ exists on  
$\{t\neq T(s)\}$ and extends smoothly up to $\{t=T(s)\}$ separately from
each side.  It suffices to show that the values from the two sides agree on 
$\{t=T(s)\}$.  Let $x_-$ denote the restriction of $x$ to 
$\{(s,t):s\in \cV, \,0\leq t\leq T(s)\}$ and 
$x_+$ the restriction of $x$ to 
$\{(s,t):s\in \cV, \,T(s)\leq t\leq t_0\}$. 
Then
$x_+(s,T(s))=x_-(s,T(s))$ for all $s$.  Differentiation 
in $s$ shows that $\p_s x_+ +\p_sT\cdot x_+'=\p_s x_- +\p_sT\cdot x_-'$
when $t=T(s)$.  Since $x_+'=x_-'$ when $t=T(s)$, it follows that also 
$\p_s x_+=\p_s x_-$ when $t=T(s)$.  
\end{proof}

Similar arguments apply for linear equations with piecewise smooth
coefficients.  In this case we do not assume continuity across the
singularity.  Consider an intial value problem 
\begin{equation}\label{linearequation}
x'(s,t) = M(s,t) x(s,t),\qquad x(s,0) = x_0(s), 
\end{equation}
where $x(s,t)\in \R^d$, $M(s,t)\in \R^{d\times d}$, and the parameter 
$s\in \R^k$. We assume that 
\[
M(s,t)=
\begin{cases}
M_-(s,t) & t< T(s)\\
M_+(s,t) & t> T(s),
\end{cases}
\]
where $T(s)>0$, $s\mapsto T(s)$ is $C^\infty$, each of $M_\pm$ is
$C^\infty$ on $\R^k\times \R$, and $x_0$ is $C^\infty$.  It is not assumed
that $M_-(s,T(s))=M_+(s,T(s))$.   We require that $x(s,t)$ is a weak
solution in the sense that that for each $s$, $x(s,t)$ is a solution for  
$t\neq T(s)$, and $x$ is continuous across $t=T(s)$. 
The proof of the  following proposition is similar to that of Proposition
\ref{smooth}: 

\begin{proposition}\label{linearprop} 
The problem \eqref{linearequation} has a unique weak solution for each $s$,
and  the restrictions of $x(s,t)$ to  
$\{(s,t):t\leq T(s)\}$ and $\{(s,t):t\geq T(s)\}$ are $C^\infty$.  
\end{proposition}
\begin{proof}
There is a unique solution $x_-(s,t)$ to 
\[
x_-'(s,t) = M_-(s,t) x_-(s,t),\qquad x_-(s,0) = x_0(s),  
\]
and $x_-$ is $C^\infty$.  Likewise, there is a unique solution $x_+(s,t)$ to 
\[
x_+'(s,t) = M_+(s,t) x_+(s,t),\qquad x_+(s,T(s)) = x_-(s,T(s)),  
\]
and $x_+$ is $C^\infty$.  The function defined by
\[
x(s,t)=
\begin{cases}
x_-(s,t) & t\leq T(s)\\
x_+(s,t) & t\geq T(s) 
\end{cases}
\]
is a weak solution of \eqref{linearequation}, and is clearly the only weak 
solution.  
\end{proof}
\medskip
To analyze continuity at $\e=0$  we will use the following two
results, which are standard applications of Gronwall's inequality.  Let
$|\cdot|$ denote the Euclidean norm on vectors, or the Euclidean operator
norm on matrices.  
\begin{proposition}\label{gronwalllinear}
Let $\cK_i:[t_0,t_1]\to \R^{d\times d}$, $i=1$, $2$, be bounded and 
measurable with $|\cK_1(t)|\leq L$, and let $f_i:[t_0,t_1]\to \R^d$,  
$i=1$, $2$, be integrable.  Let $x_i:[t_0,t_1]\to \R^d$, $i=1$, $2$, be 
continuous weak solutions to  
\[
x_i'(t)=\cK_i(t)x_i(t)+f_i(t)
\]
and set $C=\sup_{t\in [t_0,t_1]}|x_2(t)|$.  Then   
\[
\begin{split}
|x_1(t)-x_2(t)|\leq &|x_1(t_0)-x_2(t_0)|e^{L(t-t_0)}\\
&+\int_{t_0}^t\big(C|\cK_1(s)-\cK_2(s)|
+|f_1(s)-f_2(s)|\big)e^{L(t-s)}\,ds. 
\end{split}
\]
\end{proposition}

\begin{proposition}\label{gronwallnonlinear}
Let $F_i:\R^d\to \R^d$, $i=1$, $2$ be Lipschitz with constant $L$ and let
$x_i:[t_0,t_1]\to \R^d$ be $C^1$ solutions to  
\[
x_i'(t)=F_i(x_i(t)).
\]
Suppose also that $|F_1(x)-F_2(x)|\leq \delta$ for $x\in x_2([t_0,t_1])$.
Then  
\[
|x_1(t)-x_2(t)|\leq |x_1(t_0)-x_2(t_0)|e^{L(t-t_0)}
+\frac{\de}{L}\big(e^{L(t-t_0)}-1\big).
\]
\end{proposition}

Our ultimate goal in this analysis will be to understand the behavior of
stable solutions to \eqref{jacobiparallel}, \eqref{jacobiperp} as 
$\e\to 0$.  It is clear from \eqref{KY} and \eqref{eq:Kperp} that first one
needs to study $\cA_{r,\e}$ and $\r_{s,r,\e}$.  
Certainly $\cA_{r,\e}(\rho)$ is a $C^\infty$ function of $(\rho,r,\e)$ 
for $\e>0$.  Upon reducing to a first order system in the usual way,
Proposition~\ref{linearprop} implies that $\cA_{r,0}(\rho)$ and 
$\cA_{r,0}'(\rho)$ are continuous functions of $(\rho,r)$ which restrict 
to be $C^\infty$ on each of $\{r\geq \rho\}$ and $\{r\leq \rho\}$. 
The same argument as in the proof of Proposition~\ref{c1} shows 
that $\p_r\cA_{r,0}(\rho)$ is continuous across $\rho=r$, so 
that $\cA_{r,0}(\rho)$ is jointly $C^1$ everywhere.  Our ultimate interest
is in $r$ near $\pi/4$, so fix a small $\eta>0$ and set 
$\cI=[\pi/4-\eta,\pi/4+\eta]$.   

\begin{proposition}\label{Aconverge}
For $k=0,1$, 
$\p_\rho^k\cA_{r,\e}(\rho)\to \p_\rho^k\cA_{r,0}(\rho)$ 
uniformly on compact subsets of $(\r,r)\in [0,\infty)\times \cI$ as 
$\e\to 0$.  For $k\geq 2$, 
$\p_\rho^k\cA_{r,\e}(\rho)\to \p_\rho^k\cA_{r,0}(\rho)$ uniformly on
compact subsets of $([0,\infty)\times \cI)\cap \{\rho\neq r\}$ as 
$\e\to 0$.       
\end{proposition}
\begin{proof}
Reduce \eqref{Aeq} to a first order system in the usual way: set  
\[
x=\begin{pmatrix}
\cA\\
\cA' 
\end{pmatrix} \qquad 
\cK_{r,\e}=\begin{pmatrix}
0&1\\
-K^\|_{r,\e}&0
\end{pmatrix}
\]
so that \eqref{Aeq} becomes $x'=\cK x$, 
$x(0)=\begin{pmatrix} 0\\1 \end{pmatrix}$.  The first sentence follows
from Proposition~\ref{gronwalllinear} since  
$K^\|_{r,\e}-K^\|_{r,0}\to 0$ in $L^1_{loc}([0,\infty))$ uniformly in $r$.     

The convergence for $k\geq 2$ on $\{\rho<r\}$ is clear since  
$\cA_{r,\e}(\rho)$ is independent of $\e\geq 0$ on that set. 
Equation \eqref{Aformula} implies that as $\e\to 0$, eventually 
$\cA_{r,\e}$ has the form $\cA_{r,\e}(\rho)=a_+e^\rho+a_-e^{-\rho}$ on
any compact subset of $\{\rho>r\}$.  The convergence for $k=0$ implies that   
$a_\pm(r,\e)\to a_\pm(r,0)$.  The result for $k\geq 2$ therefore follows
upon differentiation in $\rho$.  
\end{proof}

We now turn to geodesics.  To streamline the notation we will often write   
$\nu=(r,\e)$, $\nu_0=(\pi/4,0)$, $\mu=(s,r,\e)$ and
$\mu_0=(0,\pi/4,0)$.  For example, we write $g_\nu:=g_{r,\e}$ or
$\g_\mu:=\g_{s,r,\e}$.  Recall that for $s\geq 0$, $\ga_\mu(t)$ denotes a
unit speed geodesic for $g_\nu$ whose distance from the origin equals $s$,
parametrized so that this minimum distance is achieved at   
$t=0$, and $\rho_\mu(t)= \rho(\ga_\mu(t))$.  
For $s>0$, $\rho_\mu(t)$ is the solution of
\eqref{rhoequation} with $\cA=\cA_\nu$ and initial conditions $\rho(0)=s$,
$\rho'(0)=0$, while  
$\rho_{0,r,\e}(t)=t$ solves the same equation but has initial 
conditions $\rho(0)=0$, $\rho'(0)=1$.  Throughout we restrict attention to
$\e$ small and $r\in \cI$, say $(r,\e)\in \cI\times [0,\e_0]$ for
fixed small positive $\e_0$.  Often we consider $s$ to be small, so    
we also fix a small $s_0>0$ and in these situations we will assume 
$s\in [0,s_0]$.  Despite the  
apparent difference in the initial conditions, \eqref{rhotsmall} shows that
$\rho_\mu(t)$ is smooth (and independent of $r$, $\e$) for   
$(t,s)\in ([0,t_0]\times [0,s_0])\setminus \{(0,0)\}$ for appropriately  
chosen $t_0$ small, and Lipschitz continuous for 
$(t,s)\in [0,t_0]\times [0,s_0]$.  
The different description of the initial conditions and the
discontinuity of the first derivatives of $\rho_\mu(t)$ at 
$(t,s)=(0,0)$ are a reflection of the singularity of polar coordinates at 
the origin.    

A first observation is that $\r_\mu(t)\to \infty$ as  
$t\to \infty$ uniformly for $(s,r,\e)\in [0,\infty)\times \cI\times
[0,\e_0]$.  In fact, \eqref{Aformula} together with the continuity of
$a_\pm$ in $(r,\e)$ established in Proposition~\ref{Aconverge} imply that
there is $a>0$ so that $\cA_\nu'(\r)/\cA_\nu(\r)\geq a$ for  
$(r,\e)\in \cI\times [0,\e_0]$ and $\r>0$.  It follows that 
$\r_\mu(t)\geq a^{-1}\log(\cosh(at))$ by the comparison argument in 
\eqref{constgeodeq}, \eqref{rhosolution}.

We will analyze \eqref{rhoequation} with $\cA=\cA_{r,\e}$ by incorporating 
$r$ as an initial value and rewriting as $x'=\cF_\e(x)$ with  
\begin{equation}\label{rhoreduce}
x=\begin{pmatrix}
\rho\\
v\\
r
\end{pmatrix},\qquad
\cF_\e(x)=\begin{pmatrix}
v\\
\frac{\cA_{r,\e}'(\rho)}{\cA_{r,\e}(\rho)}(1-v^2)\\   
0
\end{pmatrix}
\end{equation}
and with initial conditions
\[
x(0)=\begin{pmatrix}
s\\
0\\
r
\end{pmatrix}\quad\text{if}\,\,\, s>0,
\qquad
x(0)=\begin{pmatrix}
0\\
1\\
r
\end{pmatrix}\quad\text{if}\,\,\, s=0. 
\]
Our starting point is the following. 
\begin{lemma}\label{rhocontinuous}
$\r_\mu(t)$ is a continuous function of 
$(t,s,r,\e)\in [0,\infty)\times [0,\infty)\times \cI\times [0,\e_0]$.
$\rho_\mu'(t)$ and $\rho_\mu''(t)$ restrict to continuous functions on  
$\big([0,\infty)\times [0,\infty)\setminus \{(0,0)\}\big)\times \cI\times
    [0,\e_0]$.     
\end{lemma}
\begin{proof}
We have already discussed the regularity near $t=s=0$.   
It is clear that $\rho_{s,r,\e}(t)$ restricts to a $C^\infty$ function of 
$(t,s,r,\e)\in \big([0,\infty)\times [0,\infty)\setminus \{(0,0)\}\big) 
\times \cI\times (0,\e_0]$. 
Now $\cF_0$ is a locally Lipschitz function of $x$,
so $\p_t^l\rho_{s,r,0}(t)$ is a continuous function of
$(t,s,r)\in \big([0,\infty)\times [0,\infty)\setminus \{(0,0)\}\big) 
\times \cI$ for $0\leq l\leq 2$.  Proposition~\ref{Aconverge} implies that  
$\cA'_{r,\e}(\r)/\cA_{r,\e}(\r)$ converges to the corresponding function 
evaluated at $\e=0$ uniformly on compact subsets of 
$(\r,r)\in(0,\infty)\times \cI$.  The fact that  
$\p_t^l\rho_{s,r,\e}\to \p_t^l\rho_{s,r,0}$ uniformly on compact subsets of 
$(t,s,r)\in \big([0,\infty)\times [0,\infty)\setminus (0,0)\big)
\times \cI$ for $l=0$, $1$ follows from
Proposition~\ref{gronwallnonlinear}.  The convergence of $\rho''$ as 
$\e\to 0$ then follows from the differential equation \eqref{rhoequation}.    
\end{proof}

We will need to know similar continuity properties of solutions of
\eqref{jacobiparallel} and \eqref{jacobiperp} in our analysis of
$s$-derivatives of $\rho$ and in later arguments. 

\begin{lemma}\label{Xcontinuous}
Let $\cX_\mu$ be any one of $\,\cU^\prl_\mu$, $\cU^\perp_\mu$,
$\cV^\prl_\mu$, or $\cV^\perp_\mu$.  Then 
$\cX_\mu(t)$ and $\cX_\mu'(t)$ are continuous functions of   
$(t,s,r,\e)\in [0,\infty)\times [0,\infty)\times\cI\times [0,\e_0]$. 
\end{lemma}
\begin{proof}
First note that for all $(r,\e)\in \cI\times [0,\e_0]$,  
$\cX_\mu(t)=\sin(t)$ or $\cos(t)$ for
$(t,s)$ near $(0,0)$.  Rewrite \eqref{jacobiparallel} as the first
order system $x'=\cK x$ where
\begin{equation}\label{xdefine}
x=\begin{pmatrix}
\cX\\
\cX'
\end{pmatrix} \qquad 
\cK=\cK_\mu(t)=\begin{pmatrix} 
0&1\\
-K^\|_\nu(\rho_\mu(t))&0 
\end{pmatrix},
\end{equation}
and likewise for \eqref{jacobiperp}.  The functions $K^\prl_\nu(\r_\mu(t))$
and $K_\mu(t)$ are $C^\infty$ for 
$(t,s,r,\e)\in [0,\infty)\times [0,\infty)\times\cI\times (0,\e_0]$.  So 
$\cX_\mu(t)$ is also $C^\infty$ on this same set.  
The functions $K^\prl_{r,0}(\r_{s,r,0}(t))$ and $K_{s,r,0}(t)$ 
are piecewise $C^\infty$ in $(t,s,r)$ with a jump discontinuity across
$t=\ell_r(s)$.  So Proposition~\ref{linearprop} implies that  
$\cX_{s,r,0}(t)$ is also piecewise $C^\infty$ with a
jump discontinuity in second derivatives across $t=\ell_r(s)$.  
Recall from Lemma~\ref{rhocontinuous} that $\rho_{s,r,\e}$ 
and $\rho_{s,r,\e}'$ are continuous in $\e$ at $\e=0$.   
So $K^\|_{r,\e}\circ\rho_{s,r,\e}- K^\|_{r,0}\circ\rho_{s,r,0}\to 0$, 
$K_{s,r,\e}- K_{s,r,0}\to 0$ in        
$L^1_{loc}([0,\infty))$ locally uniformly in $(s,r)$.  Thus
Proposition~\ref{gronwalllinear} implies 
that $x_{s,r,\e}(t)\to x_{s,r,0}(t)$ uniformly on compact subsets of   
$[0,\infty)\times [0,\infty)\times \cI$.  
\end{proof}

Next we analyze continuity of higher derivatives of $\r_\mu$,
including $s$-derivatives.  It will suffice for our needs to restrict
attention to $s$ small, say $s\in [0,s_0]$ for $s_0>0$ small and fixed (as
above).  Set $\cR=([0,\infty)\times [0,s_0])\setminus \{(0,0)\}$.    For 
$\e=0$, the problem \eqref{rhoreduce} falls into the framework of
Proposition~\ref{smooth} with the surface $S$ given by $\rho=r$, so  
$T
\begin{pmatrix}s\\0\\r\end{pmatrix}=\ell_r(s)
=\arccos\left(\frac{\cos r}{\cos s}\right)$.
Proposition~\ref{c1} shows that    
$\rho_{s,r,0}(t)$ and $\rho_{s,r,0}'(t)$ are $C^1$ functions of 
$(t,s,r)\in \cR\times \cI$ and Proposition~\ref{smooth} implies that
$\rho_{s,r,0}(t)$ restricts to a $C^\infty$ function of $(t,s,r)$ on each
of $(\cR\times \cI)\cap \{0\leq t\leq \ell_r(s)\}$ and 
$(\cR\times \cI)\cap \{t\geq \ell_r(s)\}$. 

\begin{proposition}\label{rhoconverge}
Let $k,l\geq 0$.  If $k+l\leq 2$, then as $\e\to 0$,
$\p_s^k\p_t^l\rho_{s,r,\e}(t)$ converges to the corresponding function    
evaluated at $\e=0$, uniformly on compact subsets of 
$\cR\times \cI$.  If $k+l=3$ and $k<3$, then
$\p^k_s\p_t^l\rho_{s,r,\e}(t)$ converges to the corresponding function
evaluated at $\e =0$ 
uniformly on compact subsets of 
$(\cR \times \cI)\setminus\{t= \ell_r(s)\}$.   
\end{proposition}
\begin{proof}
The convergence for $k=0$, $0\leq l\leq 2$ is a specialization of  
Lemma~\ref{rhocontinuous}.  The stated convergence of $\rho_\mu'''$ follows 
upon differentiating \eqref{rhoequation} with respect to $t$.   

We claim that   
\begin{equation}\label{rho1}
\cA_\nu(\rho_\mu(t))\p_s\rho_\mu(t)
=\sin(s)\cU^\prl_\mu(t) 
\end{equation}
on $\cR\times \cI\times [0,\e_0]$.  To see this, one verifies directly via
the chain rule 
and the differential equations satisfied by $\cA$ and $\rho$ that     
$\cA_\nu(\rho_\mu(t))\p_s\rho_\mu(t)$ is a solution (weak
solution if $\e=0$) to \eqref{jacobiparallel}.  (For  
$\e =0$, recall that $\rho_{s,r,0}(t)$ and $\rho_{s,r,0}'(t)$ are $C^1$
functions of $(t,s,r)$.)  Now \eqref{rho1} is easily checked directly for 
$t$ near $0$ and $s\in [0,s_0]$, where we have explicit formulas for all 
involved quantities.  So the two sides are solutions of the same
differential equation which agree for $t$ small; hence they are equal.   

We use \eqref{rho1} to reduce the study of $\p_s\rho_\mu(t)$ to 
the study of $\cU^\prl_\mu(t)$.   As for the factor 
$\cA_\nu(\rho_\mu(t))$, Proposition~\ref{Aconverge} and
Lemma~\ref{rhocontinuous} 
imply that $\cA_\nu(\rho_\mu(t))\to \cA_{r,0}(\rho_{s,r,0}(t))$ and 
$\Big(\cA_\nu(\rho_\mu(t))\Big)'\to \Big(\cA_{r,0}(\rho_{s,r,0}(t))\Big)'$  
uniformly on compact subsets of $\cR\times \cI$.  So we deduce from
\eqref{rho1} and Lemma~\ref{Xcontinuous}  
that $\p_s\rho_\mu(t)\to \p_s\rho_{s,r,0}(t)$ and  
$\p_s\rho_\mu'(t)\to \p_s\rho_{s,r,0}'(t)$ uniformly on compact
subsets of $\cR\times \cI$.  
The differential equation \eqref{jacobiparallel} implies that   
$\cU^\prl_{r,s,\e}{}''\to \cU^\prl_{r,s,0}{}''$ uniformly on compact subsets
of $\big([0,\infty)\times [0,s_0]\times \cI\big)\setminus 
\{t=\ell_r(s)\}$.  Since 
$\Big(\cA_\nu(\rho_\mu(t))\Big)''\to
\Big(\cA_{r,0}(\rho_{s,r,0}(t))\Big)''$ 
uniformly on compact subsets of
$(\cR\times \cI)\setminus \{t=\ell_r(s)\}$, it follows also that 
$\p_s\rho_\mu''(t)\to \p_s\rho_{s,r,0}''(t)$ uniformly on compact 
subsets of $(\cR\times \cI)\setminus \{t=\ell_r(s)\}$.  

It remains to analyze $\p_s^2\rho_\mu(t)$ and
$\p_s^2\rho'_\mu(t)$, which we will do by  
differentiating \eqref{rho1} with respect to $s$.  Begin by considering 
$\p_s\cU^\prl_\mu$.  The equation for $\cU^\prl_\mu$ reduces 
to a first order system as in \eqref{xdefine} with $\cX=\cU^\prl_\mu$.
Define $y:=\p_s x - \frac{\p_s\rho}{\rho'}\cK x$ on 
$(0,\infty)\times [0,s_0]\times \cI\times [0,\e_0]$.  
We claim first that when $\e>0$, $y$ solves the equation  
\begin{equation}\label{yeq}
y'=\cK y +f(t),\qquad \text{where}\quad 
f(t)=-\left(\frac{\p_s\rho}{\rho'}\right)'\cK x.
\end{equation}
To see this, note that the chain rule implies  
$\rho'\p_s\cK=(\p_s\rho)\cK'$.  Then \eqref{yeq} follows by direct  
calculation:  
\begin{equation}\label{ycalc}
\begin{split}
y'=&\p_sx'-\left(\frac{\p_s\rho}{\rho'}\right)'\cK x
-\frac{\p_s\rho}{\rho'}\cK' x-\frac{\p_s\rho}{\rho'}\cK x'\\
=&\p_s(\cK x) +f(t) -(\p_s\cK) x -\frac{\p_s\rho}{\rho'}\cK^2 x\\  
=&\cK \p_s x +f(t) -\frac{\p_s\rho}{\rho'}\cK^2 x = \cK y +f(t).
\end{split}
\end{equation}
If $\e=0$, the same calculation leads to the same conclusion, but with all 
derivatives interpreted in the sense of distributions in $(s,t)$ near 
$t=\ell_r(s)$.  In particular, for $\e=0$, $y$ 
is a weak solution of \eqref{yeq}, so is continuous across 
$t=\ell_r(s)$.  The value $y(t)$ for $t$ small is independent of
$r$, $\e$.  Application of Proposition~\ref{gronwalllinear} therefore  
shows that $y_{s,r,\e}\to y_{s,r,0}$ uniformly on compact subsets of 
$(0,\infty)\times [0,s_0]\times \cI$.  The
first component of 
$y$ is $\p_s\cU_\mu^\prl - \frac{\p_s\rho}{\rho'}\cU_\mu^\prl{}'$.   
We know that  $\cU^\prl_\mu{}'\to\cU^\prl_{s,r,0}{}'$ uniformly on compact 
subsets of $[0,\infty)\times [0,s_0]\times \cI$ by
Lemma~\ref{Xcontinuous}.  So 
$\p_s\cU^\prl_\mu\to\p_s\cU^\prl_{s,r,0}$ uniformly on compact subsets of 
$(0,\infty)\times [0,s_0]\times \cI$; hence uniformly on compact subsets of  
$[0,\infty)\times [0,s_0]\times \cI$.  The second component of $y$ is    
$\p_s\cU_\mu^\prl{}'+\frac{\p_s\rho}{\rho'}(K^\|\circ\rho) \cU_\mu^\prl$.
It follows that   
$\p_s\cU^\prl_\mu{}'\to \p_s \cU^\prl_{s,r,0}{}'$ uniformly on compact subsets of 
$\big((0,\infty)\times [0,s_0]\times \cI\big) \setminus \{t= \ell_r(s)\}$;
hence uniformly on compact subsets of 
$\big([0,\infty)\times [0,s_0]\times \cI\big) \setminus \{t= \ell_r(s)\}$.    

Since
$\p_s\big(\cA_\nu(\rho_\mu(t))\big)=\cA_\nu'(\rho_\mu(t))\p_s\rho_\mu(t)$ 
converges uniformly on compact subsets of $\cR\times \cI$ to 
the corresponding expression evaluated at $\e=0$, applying $\p_s$ to
\eqref{rho1} shows that $\p_s^2\rho$ converges uniformly on compact subsets
of $\cR\times \cI$ as claimed.  Finally, 
one verifies easily via the chain rule and what we have already established
that  $\p_s\p_t\big(\cA_\nu(\rho_\mu(t))\big)$ converges
uniformly on compact subsets of $\{t\neq \ell_r(s)\}$.  So applying
$\p_s\p_t$ to \eqref{rho1} shows that $\p_s^2\rho'$ does too.   
\end{proof}

We remark that it is easily seen from the arguments above that when
$k+l=3$ and $k<3$, even though $\p_s^k\p_t^l\rho_\mu(t)$ is not 
uniformly convergent near $\{t=\ell_r(s)\}$ as $\e\to 0$, it is  
uniformly bounded near this set.  

Next consider behavior as $t\to \infty$.  
\begin{lemma}\label{rhoinfinity}
For $t$ large, $\rho_\mu$ can be written in the form 
\begin{equation}\label{rhotlarge}
\rho_\mu(t) = t + F(e^{-t},s,r,\e) 
\end{equation}
for a function $F$ satisfying 
$F\in C^{\infty}([0,1]\times [0,s_0]\times \cI\times (0,\e_0])$    
and $F|_{\e = 0}\in C^{\infty}([0,1]\times [0,s_0]\times \cI)$.  
\end{lemma}
\begin{proof}
It must be shown that the function $F$ defined by \eqref{rhotlarge} has the
stated regularity properties for $t$ large.    
The geodesic flow $\fe_t:S^*M\to S^*M$ of an asymptotically hyperbolic
metric $g$ was analyzed in \cite{2017arXiv170905053G}.  The proof of Lemma
2.7 of \cite{2017arXiv170905053G} shows 
that if $g$ is smooth and non-trapping and $u$ is a defining function for
infinity, then for 
$t\geq 0$ one can write $u(\pi(\fe_t(z))) =  e^{-t}E(e^{-t},z)$ for a
smooth positive function $E$ on $[0,1]\times S^*M$.  Here $\pi:S^*M\to M$ 
is the projection.  Note that under the change of 
variable $u=e^{-\r}$, for $\cA$ of the form \eqref{Aformula} the metric  
$g$ becomes 
$g=g^{a_+,a_-}=u^{-2}\big( du^2+ (a_++a_-u^2)^2\mathring{g}\big)$ in a
neighborhood $\cU$ of $u=0$.  First let $a_\pm$ be fixed and set 
$\widetilde{\cU}_{a_+,a_-}:=\{z\in S_{g^{a_+,a_-}}^*M:
\pi(\fe_t(z))\in \cU \text{ for all } t\geq 0\}$.    
It follows that $\rho(\pi(\fe_t(z)))=-\log u(\pi(\fe_t(z)))=t +P(e^{-t},z)$
where $P=-\log E\in C^\infty([0,1]\times \widetilde{\cU}_{a_+,a_-})$.  
To incorporate the parameters $a_\pm$, let $\mathbf{A}$ denote the set of
$(a_+,a_-)$ which arise as $(r,\e)$ varies over $\cI\times [0,\e_0]$, set  
$\mathcal{S}:=\{(a_+,a_-,z):(a_+,a_-)\in \mathbf{A}, z\in 
\widetilde{\cU}_{a_+,a_-}\}\subset \R^2\times T^*M$, 
and view $P$ as defined on $[0,1]\times \mathcal{S}$.   
The argument of the proof of Lemma 2.7 of \cite{2017arXiv170905053G} carries over
to this setting and establishes that $P$ is smooth on 
$[0,1]\times\mathcal{S}$. 

Fix $T$ large; for $t>T$ we have 
\begin{equation}\label{Ftlarge}
F(e^{-t},s,r,\e)=P\big(e^{-(t-T)}, a_+(r,\e),a_-(r,\e),\fe_T(z_s)\big) 
\end{equation}
where $z_s$ is the point (independent of $(r,\e)$) in $T^*M$
corresponding to the initial data for $\ga_\mu$ and $\fe$ denotes the
geodesic flow of $g_{r,\e}$.  
Now $a_+$, $a_-$ are $C^\infty$ functions of $r$, $\e$ for $\e>0$, and
are $C^\infty$ functions of $r$ when $\e=0$.  Likewise, $\ga_\mu(t)$ 
is $C^\infty$ in all variables for $\e>0$, and Proposition~\ref{smooth} 
implies that $\ga_{s,r,0}(t)$ is $C^\infty$ in $(s,r,t)$ for $t$ large.  
The conclusion follows.
\end{proof}

It is easily verified that for $\cA= a_+e^\rho + a_-e^{-\rho}$, one has   
\begin{equation}\label{kperpasy}
K^\perp(\rho) = -1 + e^{-2\rho}G(e^{-2\rho},a_+,a_-)
\end{equation}
with $G\in C^\infty([0,1]\times \mathbf{A})$, where, as in the proof of 
Lemma~\ref{rhoinfinity}, $\mathbf{A}$ is the 
set of all $(a_+,a_-)\in \R^2$ which arise for 
$(r,\e)\in \cI\times [0,\e_0]$.  Substituting
\eqref{rhotlarge}, \eqref{kperpasy} into \eqref{KY} and recalling that
$K^\prl_\nu(\rho_\mu(t))$ is identically $-1$ for $t$ large show that for
$t$ large,  
\begin{equation}\label{Kasy}
K_{s,r,\e}(t) = -1 +e^{-t}H(e^{-t},s,r,\e)
\end{equation}
where
\begin{equation}\label{H}
H(e^{-t},s,r,\e)=
\left(2\p_vF-e^{-t}(\p_vF)^2\right) e^{-2\r_\mu(t)}
G\big(e^{-2\r_\mu(t)},a_+(r,\e),a_-(r,\e)\big).
\end{equation}
Here $v=e^{-t}$ is the first argument of $F$ and $\p_vF$ is evaluated at 
$(e^{-t},s,r,\e)$.  The function $H$ clearly satisfies the same conditions
that $F$ satisfied in Lemma~\ref{rhoinfinity}:  
$H\in C^{\infty}([0,1]\times [0,s_0]\times \cI\times (0,\e_0])$   
and $H|_{\e = 0}\in C^{\infty}([0,1]\times [0,s_0]\times \cI)$. 

Problem 29, p. 104 of \cite{MR0069338} shows that 
there is a unique solution $\cY^\perp_\mu(t)$ to \eqref{jacobiperp} for $t$
large for which $\lim_{t\to\infty}e^t\cY(t)=1$.  Moreover, it is not hard
to show that the reasoning in the outlined solution of the cited problem 
in \cite{MR0069338} shows that $\cY^\perp_\mu(t)$ has the 
same regularity in the parameters as $K_{s,r,\e}$:   
$\cY^\perp_\mu\in C^{\infty}([T,\infty)\times [0,s_0]\times \cI\times
  (0,\e_0])$   
and $\cY^\perp_{s,r,0}\in C^{\infty}([T,\infty)\times [0,s_0]\times \cI)$
for some large $T$.  For $\e>0$, $\cY^\perp_\mu$ extends to $t\geq 0$ as a
solution with 
$\cY^\perp_\mu\in C^{\infty}([0,\infty)\times [0,s_0]\times \cI \times
  (0,\e_0])$.  For $\e=0$ we can apply  
Proposition~\ref{linearprop} backwards in time with initial data at $t=T$
to conclude that $\cY^\perp_{s,r,0}$ extends to $t\geq 0$ as a  
weak solution of \eqref{jacobiperp}, which is $C^1$ and piecewise $C^\infty$ in
$(t,s,r)$, with a jump in second derivatives across $t=\ell_r(s)$.   

Since $K^\prl_\nu(\rho_\mu(t))$ is identically $-1$ for $t$ large 
uniformly for $(s,r,\e)\in [0,s_0]\times \cI\times [0,\e_0]$, there is  
a unique solution (weak solution if $\e=0$) $\cY^\prl_\mu(t)$ to
\eqref{jacobiparallel} which equals $e^{-t}$ for $t$ large.  This solution 
$\cY^\prl_\mu$ extends backwards to $[0,\infty)$ with the same regularity
properties as $\cY^\perp_\mu(t)$.  

\begin{proposition}\label{prop:second_derivatives}
Let $0\leq l\leq 1$, $0\leq k\leq 2$ and let $\cY_\mu$ be either 
$\cY^\prl_\mu$ or $\cY^\perp_\mu$.  
As $\e\to 0$, $\p_s^k\p_t^l\cY_{s,r,\e}(t)\to \p_s^k\p_t^l\cY_{s,r,0}(t)$
uniformly on compact subsets of $[0,\infty)\times [0,s_0]\times \cI$ for
$0\leq k+l\leq 1$, and uniformly on compact subsets of  
$\big([0,\infty)\times [0,s_0]\times \cI\big) 
\setminus \{t=\ell_r(s)\}$ for $2\leq k+l\leq 3$.      
\end{proposition}
\begin{proof}
First we claim that for 
$0\leq k\leq 2$, $0\leq l\leq 1$, and for fixed large $T$,  
$\p_s^k\p_t^l\cY_{s,r,\e}(t)\to \p_s^k\p_t^l\cY_{s,r,0}(t)$ as $\e\to 0$
uniformly on $[T,\infty)\times[0,s_0]\times \cI$.  This is clear for
$\cY^\prl_\mu$ since $\cY^\prl_\mu(t)=e^{-t}$ for $t$ large.  For
$\cY^\perp_\mu$ this follows from the same argument in \cite{MR0069338}
proving the existence of $\cY^\perp_\mu$ if we establish that the function
$H$ in \eqref{Kasy} satisfies that for $0\leq k\leq 2$, 
$0\leq l\leq 1$ and $t$ large, $\p_s^k\p_t^l\big(H(e^{-t},s,r,\e)\big)$   
is uniformly bounded and continuous in $\e$ up to $\e=0$.  
Recall that $F$ is given by \eqref{Ftlarge}.  Since the
equation \ref{rhoequation} for $\r$ decouples in the equations for the geodesic flow for
$g_{r,\e}$, it is not hard to see that the argument of Lemma 2.7 of 
\cite{2017arXiv170905053G} cited in the proof of Lemma~\ref{rhoinfinity} 
applies directly to the $\r$ equation so that in \eqref{Ftlarge},
$\fe_T(z_s)$ (which amounts to $(\r,\r',\th,\th')$), can be replaced by
only $(\r_\mu(T),\r_\mu'(T))$ on the right hand side.  Since $P$ and  
$G$ are smooth, the uniform boundedness and continuity in $\e$ of 
$\p_s^k\p_t^l\big(H(e^{-t},s,r,\e)\big)$ for $t$ large follow upon
using \eqref{Ftlarge} to express $F$ in terms of $P$ in \eqref{H}, 
successively differentiating \eqref{H}, applying the chain rule, and
recalling Proposition~\ref{rhoconverge}.   

Now we use the same sort of argument as in Proposition~\ref{rhoconverge},
but backwards in time.  
We write the rest of the proof for $\cY_\mu=\cY_\mu^\perp$; the argument 
for $\cY_\mu^\prl$ is similar.  Reduce \eqref{jacobiperp} to a first 
order system $x'=\cK x$, where 
\[
x=\begin{pmatrix}
\cY\\
\cY'
\end{pmatrix} \qquad 
\cK_{s,r,\e}(t)=\begin{pmatrix}
0&1\\
-K_{s,r,\e}(t)&0
\end{pmatrix}
\]
with $K_{s,r,\e}$ defined by \eqref{KY}.  Our previous results imply that 
$K_{s,r,\e}\to K_{s,r,0}$ in $L^1([0,T])$, so 
Proposition~\ref{gronwalllinear} applied backwards in time with initial
condition at $t=T$ shows that $x_{s,r,\e}(t)\to x_{s,r,0}(t)$ uniformly on 
$[0,T]\times [0,s_0]\times \cI$.  So the convergence also holds uniformly
on $[0,\infty)\times [0,s_0]\times \cI$.  This proves the result for 
$k=0$, $0\leq l\leq 1$.  

Define $y:=\p_s x - \frac{\p_s\rho}{\rho'}\cK x$ as in the proof of
Proposition~\ref{rhoconverge}.  This time the chain rule gives 
\begin{equation}\label{K}
\frac{\p_s\rho}{\rho'} \cK' = \p_s\cK + \kappa_{s,r,\e}(t)A,
\end{equation}
where
\[
\kappa = 2\big(\rho'\p_s\rho'-(\p_s\rho) \rho''\big) 
\big(K^\|\circ\rho-K^\perp\circ\rho\big),\qquad A=\begin{pmatrix} 
0&0\\1&0\end{pmatrix}. 
\]
So the calculation analogous to 
\eqref{ycalc} via the chain rule shows that 
\begin{equation}\label{yprime}
y'=\cK y +f(t)
\end{equation}
with 
\begin{equation}\label{f}
f=-\left(\frac{\p_s\rho}{\rho'}\right)'\cK x
-\kappa Ax, 
\end{equation}
and again the equation holds weakly across $t=\ell_r(s)$ when
$\e=0$.  Since $\cK_{s,r,\e}- \cK_{s,r,0}\to 0$ and 
$f_{s,r,\e}- f_{s,r,0}\to 0$ in $L^1_{loc}((0,\infty))$ as $\e\to 0$,   
Proposition~\ref{gronwalllinear} implies that 
$y_{s,r,\e}\to y_{s,r,0}$ uniformly on compact subsets of 
$(0,\infty)\times[0,s_0]\times \cI$.  Consideration of the first component
shows that $\p_s\cY_{s,r,\e}(t)\to \p_s\cY_{s,r,0}(t)$ uniformly on compact
subsets of $(0,\infty)\times[0,s_0]\times \cI$ and consideration of the
second component shows that 
$\p_s\cY_{s,r,\e}'(t)\to \p_s\cY_{s,r,0}'(t)$ uniformly on compact
subsets of 
$\big((0,\infty)\times [0,s_0]\times \cI\big)\setminus \{t=\ell_r(s)\}$.
Since $K_\mu(t)=1$ for $t$ small uniformly in $(s,r,\e)$, the differential
equation \eqref{jacobiperp} implies that the uniform convergence    
extends down to $t=0$. This proves the result for $k=1$, $0\leq l\leq 1$.    

For $k=2$, set 
\begin{equation}\label{z}
z=\p_sy-\frac{\p_s\rho}{\rho'}\cK y
+\left(\frac{\p_s\rho}{\rho'}\right)'
\frac{\p_s\rho}{\rho'} \cK x
+\frac{\p_s\rho}{\rho'} \kappa Ax.   
\end{equation} 
We claim that $z'=\cK z +{h}(t)$, where ${h}(t)={h}_{s,r,\e}(t)$ is given by 
\[
\begin{split}
{h}(t)=&-\left(\frac{\p_s\rho}{\rho'}\right)'
\frac{\p_s\rho}{\rho'} \cK^2 x 
-\frac{\p_s\rho}{\rho'} \kappa Bx
+\cK\left[\left(\frac{\p_s\rho}{\rho'}\right)'
\frac{\p_s\rho}{\rho'} x \right]'
-\left(\frac{\p_s\rho}{\rho'}\right)'\cK y \\
&-\kappa Ay
-\frac{\p_s\rho}{\rho'}\cK f
-\p^2_{st}\left(\frac{\p_s\rho}{\rho'}\right)\cK x
-\left(\frac{\p_s\rho}{\rho'}\right)'\cK \p_sx
-\kappa A\p_sx\\
&-2\Big[\big(\p_s-\frac{\p_s\rho}{\rho'}\p_t\big)
\big(\rho'\p_s\rho'-(\p_s\rho)\rho''\big)\Big]
(K^\|\circ\rho-K^\perp\circ\rho) Ax\\ 
&+2\left(\frac{\p_s\rho}{\rho'}\right)'\kappa Ax
+\frac{\p_s\rho}{\rho'}\kappa Ax' 
\end{split}
\]
and we have set $B=\begin{pmatrix} 1&0\\0&0\end{pmatrix}$.   
Given the claim, the proof is concluded by the same sort of reasoning as
above.  Note that our previous results imply that 
${h}_{s,r,\e}\to {h}_{s,r,0}$ uniformly on compact subsets of  
$\big((0,\infty)\times [0,s_0]\times \cI\big)\setminus 
\{t=\ell_r(s)\}$,    
and ${h}_{s,r,\e}$ is uniformly bounded on compact subsets of 
$(0,\infty)\times [0,s_0]\times \cI$.  So $\cK$ and $h$ converge in 
$L^1_{loc}((0,\infty))$.  Thus Proposition~\ref{gronwalllinear} 
shows that $z_{s,r,\e}\to z_{s,r,0}$ uniformly on compact subsets of 
$(0,\infty)\times [0,s_0]\times \cI$.  According to \eqref{z}, 
$z-\p_sy$ is the sum of three terms, each of which converges uniformly on
compact subsets of 
$\big((0,\infty)\times [0,s_0]\times \cI\big)\setminus 
\{t=\ell_r(s)\}$.  So 
$\p_sy$ also converges uniformly on compact subsets of 
$\big((0,\infty)\times [0,s_0]\times \cI\big)
\setminus \{t=\ell_r(s)\}$.
And $\p_s y-\p_s^2x=-\p_s\Big(\frac{\p_s\rho}{\rho'}\cK x\Big)$ 
converges uniformly on compact subsets of 
$\big((0,\infty)\times [0,s_0]\times \cI\big)
\setminus \{t=\ell_r(s)\}$, so $\p_s^2x$ does too.  Again the differential
equation \eqref{jacobiperp} implies that the uniform convergence extends
to $t=0$.   

The proof that $z'=\cK z +{h}(t)$ is a calculation similar to  
\eqref{ycalc}, \eqref{yprime} but involving more terms.  Differentiate 
\eqref{z} with respect to $t$, expand the differentiations
using the Leibnitz rule, substitute \eqref{yprime} for the two
occurrences of $y'$ and \eqref{K} for the two occurrences of 
$\frac{\p_s\rho}{\rho'} \cK'$ on the right-hand side, and collect terms.
One obtains   
\[
\begin{split}
z'=&\cK\left(\p_s y -\frac{\p_s\rho}{\rho'}\cK y\right) +\p_s f
-\left(\frac{\p_s\rho}{\rho'}\right)'\cK y
-\ka A y 
-\frac{\p_s\rho}{\rho'}\cK f \\
&+\left(\frac{\p_s\rho}{\rho'}\right)'(\p_s\cK)x 
+2\left(\frac{\p_s\rho}{\rho'}\right)'\ka Ax
+\cK\left[\left(\frac{\p_s\rho}{\rho'}\right)'\frac{\p_s\rho}{\rho'}x\right]' \\
&+\frac{\p_s\rho}{\rho'}\ka'A x
 +\frac{\p_s\rho}{\rho'}\ka Ax'.  
\end{split}
\]
Now substitute 
\[
\p_s y -\frac{\p_s\rho}{\rho'}\cK y = z
-\left(\frac{\p_s\rho}{\rho'}\right)'
\frac{\p_s\rho}{\rho'} \cK x
-\frac{\p_s\rho}{\rho'} \kappa Ax
\]
from \eqref{z} in the first term on the right-hand side, expand
$\p_sf$ by differentiating \eqref{f}, and compare terms to obtain   
\[
\begin{split}
z'&-\cK z - h(t) \\
=&\left[\left(\frac{\p_s\rho}{\rho'} \ka'-\p_s\ka\right)
+2\Big[\big(\p_s-\frac{\p_s\rho}{\rho'}\p_t\big)
\big(\rho'\p_s\rho'-(\p_s\rho)\rho''\big)\Big]
(K^\|-K^\perp)\circ\rho\right]Ax.   
\end{split}
\]
Finally, observe that the right-hand side vanishes.    
\end{proof}

It is easily checked that for $\mu_0=(0,\pi/4,0)$ the decaying solution 
$\cY_{\mu_0}^\prl=\cY_{\mu_0}^\perp=:\cY_{\mu_0}$ is given by
$\cY_{0,\pi/4,0}(t)=
\begin{cases} \sqrt{2}e^{-\pi/4}\cos(t), & 0\leq t\leq \pi/4\\ 
e^{-t},& t\geq \pi/4
\end{cases}$.  
Since $\cY_{\mu_0}(0)>0$, it follows from continuity that 
$\cY_{\mu}^\prl(0)>0$, $\cY_\mu^\perp(0)>0$ for all $\mu$ sufficiently
close to $\mu_0$.  For such $\mu$ we define 
$\cW_\mu^\prl(t)=\cY_\mu^\prl(t)/\cY_\mu^\prl(0)$ and 
$\cW_\mu^\perp(t)=\cY_\mu^\perp(t)/\cY_\mu^\perp(0)$
so that $\cW_\mu^\prl$, $\cW_\mu^\perp$ are the decaying solutions
satisfying $\cW_\mu(0)=1$.  These $\cW_\mu$ inherit the continuity
properties of $\cY_\mu$ stated in
Proposition~\ref{prop:second_derivatives}.   

In the sequel we will need the following lemma.
\begin{lemma}\label{Knegative}
There exists $T>0$ so that
$K_\nu^\prl(\r_\mu(t))<0$ and $K_\mu(t)<0$ for 
$(t,s,r,\e)\in [T,\infty)\times [0,\infty)\times \cI\times [0,\e_0]$.    
\end{lemma}
\begin{proof}
Since $\r_\mu(t)\to \infty$ as $t\to \infty$ uniformly for 
$\mu=(s,r,\e)\in [0,\infty)\times \cI\times [0,\e_0]$ and $K_\mu(t)$ is a   
convex combination of $K_\nu^\|(\r_\mu(t))$ and $K_\nu^\perp(\r_\mu(t))$,
it suffices to show that there exists $\r_0$ independent of 
$\nu=(r,\e)\in \cI\times [0,\e_0]$ so that $K^\prl_\nu(\rho)<0$ and 
$K^\perp_\nu(\rho)<0$ for $\r\geq \r_0$.
For $K_\nu^\prl$ this is clear since $K_\nu^\prl(\r)=-1$ for $\r>r+\e$.    
Equation \eqref{Aformula} and the continuity of $a_\pm$ in $(r,\e)$ show
that we can choose $\r_0$ independent of  
$\nu\in \cI\times [0,\e_0]$ so that $\cA_\nu'(\r_0)>1$.  The
differential equation for $\cA_\nu$ implies that 
$\cA_\nu'(\r)\geq \cA_\nu'(\r_0)>1$ for $\r\geq \r_0$.  Then
$K_\nu^\perp(\r)=\cA_\nu^{-2}(\r)\big(1-(\cA_\nu'(\r))^2\big)<0$ as
desired. 
\end{proof}

\begin{lemma}
$\cW_\mu^\prl(t)>0$ and $\cW_\mu^\perp(t)>0$ for all $t\geq 0$ and for all
$\mu$ sufficiently near $\mu_0$.  
\end{lemma}
\begin{proof}
We suppress the superscripts ${}^\prl$, ${}^\perp$; the argument is the
same for both.  Recall the solution $\cV_\mu$ with initial conditions
$\cV_\mu(0)=0$, $\cV_\mu'(0)=1$.  The Wronskian 
$\cW_\mu\cV'_\mu -\cW_\mu'\cV_\mu =1$.  We will show below that
$\cV_\mu(t)>0$ for all $t>0$.  Given this, it follows that $\cW_\mu'(t)<0$
at every $t$ for which $\cW_\mu(t)=0$.  The vanishing of $\cW_\mu(t)$ for
any $t$ is therefore inconsistent with the fact that  
$\cW_\mu$ is asymptotic to a positive multiple of $e^{-t}$ as $t\to
\infty$.   

Now we show that $\cV_\mu(t)>0$ for all $t>0$ for $\mu$ sufficiently close  
to $\mu_0$.  $\cV_{\mu_0}$ is identified in \eqref{Vformula} (take
$\ell(s)=\pi/4$) and clearly is positive on $(0,\infty)$.  Choose $T$ as in
Lemma~\ref{Knegative}.  Continuity (from Lemma~\ref{Xcontinuous})  
and the fact that $\cV_\mu(t)=\sin (t)$ for $t$ small imply that 
there is a neighborhood of 
$\mu_0$ for which $\cV_\mu>0$ and $\cV_\mu'>0$ on $(0,T]$.
The differential equation \eqref{jacobiparallel} or \eqref{jacobiperp}
implies that $\cV_\mu>0$ on $[T,\infty)$ as desired.   
\end{proof}

\begin{proposition}\label{prop_twice_differentiable}
Let $\cW_\mu(t)$ be either $\cW_\mu^\prl(t)$ or $\cW_\mu^\perp(t)$.
Then $\p_s( \cW'_{\mu}(0))\big|_{s=0}=0$ and there exist
a neighborhood $U$ of $(\pi/4,0)$ and $\sigma>0$ such that if   
$(r,\e)\in U$ and $0\leq s\leq \sigma$, then
$\p_s^2(\cW'_{\mu}(0))<0$. 
\end{proposition}
\begin{proof}
For the first statement we actually show 
$\p_s(\cW_{\mu}(t))\big|_{s=0}=0$ for all $t$.  
It is clear from \eqref{rho1} that
$\p_s\r\big|_{s=0}=\p_s\r'\big|_{s=0}=0$.  
In case $\cW_\mu=\cW_\mu^\perp$, if $\e>0$ differentiation of
\eqref{jacobiperp} shows that $\p_s \cW_\mu|_{s=0}$ is also a
solution of \eqref{jacobiperp}.  This holds in the weak sense  
when $\e=0$ by the reasoning in the 
proof of Proposition~\ref{prop:second_derivatives}.  Since
$\p_s \cW_\mu|_{s=0}$ vanishes as $t\to \infty$, it must be a multiple of
$\cW_\mu$.  Since $\p_s \cW_\mu(0)|_{s=0}=0$ and $\cW_\mu(0)=1$, 
the multiple must be zero.  The same argument applies to
$\cW_\mu=\cW_\mu^\prl$ upon differentiation of
\eqref{jacobiparallel}.  

For the second statement, it suffices to show that
$\p_s^2(\cW'_{s,\pi/4,0}(0))|_{s=0}<0$ by
Proposition~\ref{prop:second_derivatives}. 
Again consider first $\cW=\cW^\perp$ and suppress writing 
${}^\perp$ on all quantities below.  For $s$ small we can write 
$\cW_s:=\cW_{s,\pi/4,0}$ as a linear combination of the solutions $\cU_s$
and $\cV_s$ 
given by \eqref{190519solutionz}, \eqref{190519solution}.  By first
considering the asymptotics as $t\to \infty$ and then the value at $t=0$,
one finds that for $s>0$ small 
\[
\cW_s = \cU_s -\csc(s)\cot(\Theta_\infty(s))\cV_s
\]
where $\Theta_\infty$ is given by \eqref{Thetainfinity}.  Hence
$\cW_s'(0)=-\csc(s)\cot(\Theta_\infty(s))$.
Evaluation of \eqref{Thetainfinity} gives $\Theta_\infty(0)=\pi/2$ and
\eqref{derivTheta} shows that $\p_s\Theta_\infty=-s^2/2+O(s^3)$.  
Thus 
$\Theta_\infty(s)= \pi/2-s^3/6 + O(s^4)$ so that
$\cot(\Theta_\infty(s))= s^3/6 +O(s^4)$.  This gives  
$\p_s^2\cW_s'(0)|_{s=0}=-1/3$ as desired.  For the second case   
$\cW=\cW^\prl$, write $\cW_s$ as a linear combination of the solutions 
\eqref{Ztildeformula} and \eqref{Vformula} and find, also using \eqref{ell},
\[
\cW_s = \cU_s -\frac{1-\sqrt{\cos(2s)}}{1+\sqrt{\cos(2s)}}\cV_s,
\]
so that $\cW_s'(0)=-\frac{1-\sqrt{\cos(2s)}}{1+\sqrt{\cos(2s)}}$.    
This time there are no indeterminants and one finds without difficulty
$\p_s^2\cW_s'(0)|_{s=0}=-1$.  
\end{proof}

We will use the next proposition to rule out interior conjugate points for
$s$ near 0.  

\begin{proposition}\label{123018proposition}
Let $f\in L^1(\R)$ be an even function and suppose $\cW$ is a $C^1$
weak solution to 
\begin{equation}\label{110318eq1}
\cW''(t)+(-1+f(t))\cW(t)=0
\end{equation}
with $\cW(t)> 0$ for $t\geq 0$ and $\lim_{t\to\infty} \cW(t)=0$.    
There are no nontrivial solutions of \eqref{110318eq1} vanishing 
at two distinct values of $t$ if and only if $\cW'(0)\leq 0$.   
\end{proposition}
\begin{proof}
	First assume that $\cW'(0)\leq 0$.  By the Sturm Separation
        Theorem, all solutions of \eqref{110318eq1} vanish at most once if
        there         exists one solution of \eqref{110318eq1} that never
        vanishes, so it is enough to show that $\cW(t)\neq 0$ for all $t$.
        If $\cW'(0)=0$ 
        then $\cW(|t|)$ is a non-vanishing $C^1$ solution of
        \eqref{110318eq1}, so suppose $\cW'(0)<0$ and let $t_0<0$ be such
        that $\cW(t_0)=0$ for the sake of contradiction.  Define
        $\cV:\R\to\R$ by $\cV(t)=\cW(t)-\cW(-t)$.  Then $\cV$ satisfies   
        \eqref{110318eq1} by evenness of $f$.  Since $\cV(-t_0)>0$, %
        $\cV(0)=0$ and $\cV'(0)<0$, there exists $0<t_1<-t_0$ with
        $\cV(t_1)=0$.  This contradicts the Sturm Separation Theorem, since
        $\cW$ and $\cV$ are linearly independent and $\cW>0$ on $[0,-t_1]$.

	For the converse, suppose that no solutions of \eqref{110318eq1}
        vanish twice and $\cW'(0)>0$.  We can normalize to assume
        $\cW(0)=1$.  Let $\cU$ denote the solution of
        \eqref{110318eq1} with $\cU(0)=1$ and $\cU'(0)=0$, which is even by
        evenness of $f$.  We claim that $\cU$ vanishes for some positive
        $t$, hence twice.  Suppose this is not the case, i.e. $\cU>0$ on
        $\R$.
	We have $0<\cU(t)<\cW(t)$ for all $t>0$; otherwise the function  
        $\cW-\cU$ would vanish at least twice on $[0,\infty)$. 
	We conclude that $\lim_{t\to\infty}\cU(t)=0$ and hence all
        solutions of \eqref{110318eq1} decay as $t\to \infty$. 
	This is a contradiction:  Problem 29 in p. 104 of \cite{MR0069338}
        implies that there are solutions of \eqref{110318eq1} which grow
        exponentially as $t\to\infty$. 
\end{proof}

Finally we can prove Propositions~\ref{exist}, \ref{noicpssmall} and 
\ref{noicpslarge}.  

\begin{proof}[Proof of Proposition~\ref{exist}]
It is easy to check that for $g_{r,0}$ the decaying solution 
on radial geodesics is given for $t\geq 0$ by 
\[
\cY_{0,r,0}(t)=
\begin{cases}
e^{-r}\big(\cos(t-r)-\sin(t-r)\big) & \quad t\leq r\\
e^{-t} & \quad t\geq r 
\end{cases}.
\]
(Since $\cY^\prl_{0,r,\e}=\cY^\perp_{0,r,\e}$, we
suppress the ${}^\prl$.)  So  
$\cY_{0,r,0}'(0)=e^{-r}\big(\sin(r)-\cos(r)\big)$.  If
$r_1<\pi/4<r_2$, then $\cY_{0,r_1,0}'(0)<0<\cY_{0,r_2,0}'(0)$ and we can 
choose $r_1$ and $r_2$ as close to $\pi/4$ as we like.  Continuity 
(from Proposition~\ref{prop:second_derivatives}) implies
that if $\e$ is small enough, then also 
$\cY_{0,r_1,\e}'(0)<0<\cY_{0,r_2,\e}'(0)$.  The mean 
value theorem gives the existence of $r$, $r_1<r<r_2$, with
$\cY_{0,r,\e}'(0)=0$.   
\end{proof}

\begin{proof}[Proof of Proposition~\ref{noicpssmall}]
Choose $\sigma$ and $U$ so that the conclusion $\p_s^2(\cW'_{\mu}(0))<0$
for $(r,\e)\in U$ and $0\leq s\leq\sigma$ of 
Proposition~\ref{prop_twice_differentiable} holds for both $\cW^\prl_\mu$
and $\cW^\perp_\mu$.  
The hypothesis $\cY^\prl_{0,r,\e}{}'(0)=0$ certainly implies that 
$\cW^\prl_{0,r,\e}{}'(0)=0$, and also we have 
$\cW^\perp_{0,r,\e}{}'(0)=0$ since $\cW^\prl_{0,r,\e}=\cW^\perp_{0,r,\e}$.  
Combining this with
$\p_s\big(\cW_{s,r,\e}'(0)\big)\big|_{s=0}=0$, it 
follows that $\cW_{s,r,\e}'(0)\leq 0$ for $0\leq s\leq \sigma$ for both 
$\cW^\prl_\mu$ and $\cW^\perp_\mu$.  
Proposition~\ref{123018proposition} then implies that along any geodesic
$\g_\mu\subset \Sigma_\g$ with $0\leq s\leq \sigma$, no nontrivial normal  
Jacobi field which is either tangent to $\Sigma_\g$ or normal to 
$\Sigma_\g$ can vanish twice.  Proposition~\ref{propositiontotallygeodesic}
shows that no nontrivial normal Jacobi field can vanish twice, just as in
the proof for $g_{\pi/4,0}$.  Hence 
$g_{r,\e}$ has no interior conjugate points on a geodesic $\g_\mu$ for
which $0\leq s\leq \sigma$.   
\end{proof}

\begin{proof}[Proof of Proposition~\ref{noicpslarge}]
First we claim that there exists $S>0$ so that for any 
$(r,\e)\in \cI\times [0,\e_0]$, 
$g_{r,\e}$ has no interior conjugate points on any geodesic $\g_{s,r,\e}$
with $s\geq S$.  To see this, recall that we showed in the proof of  
Lemma~\ref{Knegative} that there is $\rho_0>0$ independent of 
$\nu\in \cI\times [0,\e_0]$ so that $K^\prl_\nu(\rho)<0$ and
$K^\perp_\nu(\rho)<0$ for $\r\geq \r_0$.  Since for any $(r,\e)$, 
$s=\min_{t\in \R}\r_{s,r,\e}(t)$, we know that if  
$s\geq \rho_0$, then $\rho_\mu(t)\geq \rho_0$ for all $t\in \R$.  
It follows that $K^\prl_\nu(\r_\mu(t))<0$ and $K_\mu(t)<0$ for $t\in \R$ so 
long as $s\geq \r_0$ and $(r,\e)\in \cI\times [0,\e_0]$.  Since the
equation $Y''=0$ has a 
nonvanishing solution on $\R$, the Sturm Comparison Theorem implies that
if $s\geq \r_0$ and $(r,\e)\in \cI\times [0,\e_0]$, then no nontrivial
solution of \eqref{jacobiparallel} or \eqref{jacobiperp} can vanish
twice.  This gives the claim with $S=\r_0$ upon recalling  
Proposition~\ref{propositiontotallygeodesic}.  

We will now show that given any $\sigma>0$, there is a neighborhood $V$
of $(\pi/4,0)$ such 
that $\cU^\prl_\mu(t)$ and $\cU^\perp_\mu(t)$ are positive for 
all $t\in \R$ for $(r,\e)\in V$ and $\sigma\leq s\leq S$, thus 
excluding nontrivial solutions vanishing twice by the Sturm Separation
Theorem.  It suffices to consider $t\geq 0$ since $\cU^\prl_\mu(t)$ and
$\cU^\perp_\mu(t)$ are even.  Choose $T$ as in Lemma~\ref{Knegative}.  We
showed in the proofs of Lemmas~\ref{122918gulliver} and  \ref{121618lemma}
that 
$\cU^\prl_{s,\pi/4,0}(t)$ and $\cU^\perp_{s,\pi/4,0}(t)$ are everywhere
positive for any $s\geq 0$, and that analysis also shows that these
solutions grow exponentially as $t\to \infty$ uniformly for 
$s\in [\sigma,S]$.  Increasing $T$ if necessary, continuity (from 
Lemma~\ref{Xcontinuous}) 
implies that there is a neighborhood $V$ of $(\pi/4,0)$ and $c>0$ so that  
$\cU^\prl_\mu(t)\geq c$, $\cU^\perp_\mu(t)\geq c$ for 
$0\leq t\leq T$, $(r,\e)\in V$ and $0\leq s\leq S$, and also  
$\cU^\prl_\mu{}'(T)>0$, $\cU^\perp_\mu{}'(T)>0$ for 
$(r,\e)\in V$ and $\sigma\leq s\leq S$.  The differential equations
satisfied by $\cU^\prl_\mu$ and $\cU^\perp_\mu$ then imply that the
solutions stay positive for $t>T$.     
\end{proof}

\bibliographystyle{alpha}
\bibliography{mybib.bib}

\end{document}